\documentclass[preprint,12pt]{elsarticle}

\makeatletter

\def\ps@pprintTitle{%
    \let\@oddhead\@empty
    \let\@evenhead\@empty
    \def\@oddfoot
    {\hbox to \textwidth%
        {\ifnopreprintline\relax\else
                \@myfooterfont%
                \ifx\@elsarticlemyfooteralign\@elsarticlemyfooteraligncenter%
                    \hfil\@elsarticlemyfooter\hfil%
                \else%
                    \ifx\@elsarticlemyfooteralign\@elsarticlemyfooteralignleft%
                        \@elsarticlemyfooter\hfill{}%
                    \else%
                        \ifx\@elsarticlemyfooteralign\@elsarticlemyfooteralignright%
                            {}\hfill\@elsarticlemyfooter%
                        \else%
                            Preprint submitted to \ifx\@journal\@empty%
                                Elsevier%
                            \else\@journal\fi\hfill\@date\fi%
                    \fi%
                \fi%
            \fi%
        }%
    }%
    \let\@evenfoot\@oddfoot}

\ifx \enableCTEX \undefined
\else
    \RequirePackage{ctex}  
\fi

\usepackage[english]{babel}                                       
\usepackage[utf8]{inputenc}                                       
\usepackage[T1]{fontenc}                                          
\usepackage{amsmath, amsthm, amssymb, amscd, amsfonts, mathtools} 
\usepackage{indentfirst}                                          
\usepackage{graphicx}                                             
\usepackage{xcolor}                                               
\usepackage{tikz}                                                 
\usepackage{bm}                                                   
\usepackage{multirow}                                             
\usepackage{lineno}                                               
\usepackage{caption}                                              
\usepackage{subcaption}                                           
\usepackage[shortlabels]{enumitem}                                
\usepackage{chngcntr}                                             
\usepackage{color}
\usepackage{hyperref}                                             
\usepackage{algorithm,algorithmic}
\usepackage[numbers]{natbib}
\usepackage{float}
\usepackage{booktabs} 

\journal{(Journal Name)}
\usetikzlibrary{shapes,arrows,chains}
\AtBeginDocument{\def\@citecolor{red}}
\AtBeginDocument{\def\@linkcolor{blue}}
\hypersetup{colorlinks}

\theoremstyle{definition}

\newtheorem*{definition*}{Definition}
\newtheorem*{example*}{Example}
\newtheorem{remark}{Remark}[section]
\newtheorem*{remark*}{Remark}

\theoremstyle{plain}
\newtheorem{theorem}{Theorem}[section]

\newtheorem{lemma}[theorem]{Lemma}
\newtheorem{proposition}[theorem]{Proposition}

\newtheorem*{axiom*}{Axiom}
\newtheorem*{exercise*}{Exercise}
\newtheorem*{lemma*}{Lemma}
\newtheorem*{proposition*}{Proposition}
\newtheorem*{theorem*}{Theorem}
\newtheorem*{corollary*}{Corollary}

\counterwithin{figure}{section}
\counterwithin{table}{section}
\numberwithin{equation}{section}
\numberwithin{algorithm}{section}

\newcounter{eq_count}
\newcounter{subeq_count}
\let\OldSubequations\subequations
\renewcommand{\subequations}{
    \setcounter{subeq_count}{0}
    \OldSubequations
}

\newcounter{thm_count}
\newcounter{fig_count}
\newcounter{tab_count}
\newcounter{alg_count}
\let\OldSection\section
\renewcommand{\section}{
    \setcounter{eq_count}{0}
    \setcounter{thm_count}{0}
    \setcounter{fig_count}{0}
    \setcounter{tab_count}{0}
    \setcounter{alg_count}{0}
    \OldSection
}

\makeatother

\newcommand{\half}{\frac{1}{2}}

\newcommand{\vr}{\mathbf{x}}
\newcommand{\vOmega}{\mathbf{\Omega}}

\newcommand{\vx}{\mathbf{x}}
\newcommand{\mfx}{{\mathbf{X}}}

\newcommand{\dd}{{\text{d}}}

\newcommand{\pdrv}[2]{\frac{\partial{#1}}{\partial{#2}}}

\newcommand{\restr}[2]{\left.\kern-\nulldelimiterspace #1 \vphantom{\big|} \right|_{#2}}
\newcommand{\qc}{\text{,}\quad}

\begin{document}

\begin{frontmatter}

    \title{An efficient asymptotic preserving Monte Carlo
method for frequency-dependent radiative transfer equations}
     \author[1]{Yiyang Hong}
     \ead{hongyy23@stu.xmu.edu.cn}

      \author[2]{Yi Shi}
      \ead{shiyi@sdu.edu.cn}

      \author[1]{Yi Cai}
      \ead{yicaim@stu.xmu.edu.cn}

     \author[3]{Tao Xiong\corref{cor1}}
     \ead{taoxiong@ustc.edu.cn}

     \address[1]{School of Mathematical Sciences, Xiamen University, Xiamen, Fujian 361005, PR China}

      \address[2]{School of Mathematics, Shandong University, Jinan, Shandong 250100, PR China}
    
     \address[3]{School of Mathematical Sciences, University of Science and Technology of China, Hefei, Anhui 230026, PR China}

      \cortext[cor1]{Corresponding author.}

    \begin{abstract}
        In this paper, we develop an efficient asymptotic-preserving (AP) Monte Carlo (MC) method for frequency-dependent radiative transfer equations (RTEs), which is based on the AP MC method proposed for the gray RTEs in \cite{shi2023efficient}. We follow the characteristics-based approach by Zhang et al. \cite{zhang2023asymptotic} to get a reformulated model, which couples a low dimension convection-diffusion-type equation for the macroscopic quantities with a high dimension transport equation for the radiative intensity. To recover the correct free streaming limit due to frequency-dependency, we propose a correction to the reformulated macroscopic equation. The macroscopic system is solved using a hybrid method: convective fluxes are handled by a particle-based MC method, while diffusive fluxes are treated implicitly with central difference. 
        To address the nonlinear coupling across multiple frequency groups, we adopt a Picard iteration with a predictor-corrector procedure, which decouples a global nonlinear system into a \textit{space-only linear}  system with \textit{spatially decoupled scalar nonlinear}  equations. Once the macroscopic update is done, the transport equation is efficiently solved via a MC  method using the implicitly defined but known emission source. This approach enables larger time steps independent of the speed of light, significantly enhancing computational efficiency, especially for frequency-dependent RTEs. Formal AP analysis in the diffusive scaling is established.  Numerical experiments are performed to demonstrate the high efficiency and AP property of the proposed method.


    \end{abstract}


\end{frontmatter}


\section{Introduction}
\label{sec:introduction}
The radiative transfer equations (RTEs) are fundamental in modeling photon transport and interactions with matter in high-energy-density systems, such as astrophysics, inertial confinement fusion (ICF) and high-temperature flows. Their solution is challenging due to the high dimensionality, strong nonlinear coupling, and multiscale behavior in both space and time. It has attracted a lot of attention for numerical studies because of its importance but high complexity.

A popular strategy for simulating the RTEs is the deterministic method, which includes the spherical harmonic (also known as the $P_N$) and the discrete ordinate (also known as the $S_N$) methods. In the $P_N$ approximation \cite{pomraning2005equations,reed1972spherical,mcclarren2008solutions,mcclarren2010robust}, the radiation intensity is decomposed into a series of angular moments to arrive at a finite hyperbolic system. In the $S_N$ approximation \cite{pomraning2005equations,reed1972spherical,adams1997subcell,sun2015asymptoticB,xiong2022high}, the intensity is divided into certain selected directions, and a system of coupled discrete equations should be solved.  However, each method has its limitations \cite{brunner2002forms}: the $S_N$ method suffers from \textit{ray effects}, while the $P_N$ method is prone to \textit{wave effects} in time-dependent problems. 

An alternative approach is the stochastic method, such as the Monte Carlo (MC) method. MC methods allow for continuous treatment of phase-space variables during particle tracking. Despite inherent statistical fluctuations, they are highly flexible, especially for complex geometries, and are free from ray effects or wave effects. The implicit Monte Carlo (IMC)  method \cite{fleck1971implicit} is an important stochastic approach for solving RTEs. 
 It treats absorption and emission semi-implicitly via effective scattering, improving stability, and allowing for larger time steps compared to fully explicit MC methods. However, in optically thick regime with short photon mean free paths, particles undergo many collisions due to effective scattering, resulting in high computational cost and reduced efficiency.  Efforts have been made to improve the efficiency of the IMC method in optically thick regimes, such as the random walk approach \cite{fleck1984random,giorla1987random,keady2017improved:}, discrete diffusion Monte Carlo (DDMC) methods 
 \cite{densmore2007hybrid,densmore2012hybrid:},
 and implicit Monte Carlo diffusion (IMD) methods \cite{cleveland2010extension,gentile2001implicit,cleveland2014mitigating}. The DDMC methods and IMD methods are transport-diffusion hybrid methods which simulate the RTEs with a diffusion approximation in optically thick regimes and the standard IMC method in other regimes. Special efforts need to be made for domain decomposition and information exchange at transport-diffusion interfaces \cite{densmore2006interface}.

Gray radiative transfer equations (GRTEs) are a simplified form of RTEs, where the opacity depends solely on the material type or material temperature. This enables spatial partitioning into optically thick and thin regimes, allowing for the use of different numerical methods in each. In frequency-dependent radiative transfer equations (FRTEs), opacity typically decreases with frequency, making a regime optically thick at low frequencies but thin at high ones. This frequency dependence presents significant mathematical and numerical challenges, particularly in consistently coupling diffusion-like and free-streaming behaviors across the entire computational domain. 
One effective approach to address these challenges is the moment-based acceleration scheme known as the High-Order/Low-Order (HOLO) algorithm \cite{chacon2017multiscale,yee2017stable,park2019multigroup,bolding2017high,ZHANG2023112368,feng2025decomposed}. In the HOLO algorithm, a LO system - comprising the first two moments of the RTEs coupled with the material temperature equation is solved first.  The solution from this LO system is then used to determine the Planckian emission source term, which is subsequently employed to solve the HO equations. The benefit of HOLO algorithms lies in their ability to eliminate nonlinearity, however, ensuring long-term accuracy and nonlinear stability requires maintaining discrete consistency between the HO and LO formulations \cite{chacon2017multiscale}.
Another efficient approach to deal with these difficulties is the asymptotic preserving (AP) scheme. 
This approach has its origin in capturing steady-state solution for neutron transport in the diffusive regime \cite{larsen1987asymptotic,larsen1989asymptotic}, and later applied to non-stationary transport problems \cite{klar1998asymptotic}. The basic idea behind AP schemes is to ensure that the numerical method accurately captures the asymptotic limit of the mathematical model at the discrete level \cite{jin2022asymptotic}, avoiding dealing with the transport-diffusion interfaces in the domain decomposition method. There are various ways to construct an AP scheme for RTEs, including micro-macro decomposition \cite{shi2023efficient,xiong2022high}, unified gas kinetic (UGK) method \cite{sun2015asymptoticB,shi2020asymptotic,sun2015asymptoticA,li2020unified,liu2023implicit,yang2025unified,xu2021positive,xu2024spatial}, three-state update method \cite{tang2021accurate}, AP-HOLO \cite{ZHANG2023112368,feng2025decomposed,park2020toward}, linear-discontinuous spatial differencing scheme \cite{morel1996linear}, etc and some of them have been extended to FRTEs.

In this work, we aim to develop an efficient AP MC method for FRTEs. For the frequency variable, we employ the classical multi-group approach \cite{pomraning2005equations}. In this approach, the frequency variable is discrezed into a finite number of groups and the frequency integration is only performed over the groups. Building on the idea originally proposed in the UGK scheme \cite{xu2010unified, mieussens2013asymptotic}, we substitute, for each group, an integral solution of the microscopic transport equation to construct the flux for the corresponding macroscopic equation.  Unlike UGK-type methods, which numerically approximate the macroscopic radiative flux based on information from neighboring cells, we adopt a characteristics-based approach \cite{shi2023efficient,zhang2023asymptotic,cai2024asymptotic}, wherein we further perform integration by parts and suitably approximate the resulting integrals along the characteristic lines to close the macroscopic flux. This approximation can be formally shown to be of order $O(\Delta t^2)$ for linear kinetic model \cite{zhang2023asymptotic,cai2024asymptotic}. The resulting macroscopic equation incorporates long-range characteristic tracing, which enables the use of larger time steps independent of the speed of light. To address model approximation errors in the free-streaming regime, we introduce a correction to the original model. We then adopt a hybrid finite volume method for solving the macroscopic system: the convective flux is computed using a particle-based MC method, while the diffusive flux is handled implicitly with central difference.
The new challenge for FRTEs is the nonlinear coupling between the group-wise radiative intensity and the group-wise Planck function, which results in a high-dimensional implicit system that involves not only spatial variables but also the frequency (group) dimension.
 Instead of employing the linearized iterative solver from \cite{sun2015asymptoticB,sun2015asymptoticA}, as used in our previous work \cite{shi2023efficient}, we adopt a Picard iteration  with the predictor-corrector procedure proposed in \cite{tang2021accurate} to efficiently manage this group-wise nonlinearity.
 The solution of the global nonlinear system is divided into two stages: a predictor step that solves a space-only linear system, followed by a corrector step that solves  scalar nonlinear equations locally within each spatial cell. This predictor-corrector process is iterated until convergence within each Picard iteration.
 Notably, in our scheme, the linear system involves a matrix whose size depends only on the number of spatial cells, allowing for efficient computation.
 Once the macroscopic equation is solved, the updated material temperature provides an a priori estimate for the Planckian emission source. This reduces the transport equation to a purely absorbing problem with a known source, which can be directly solved using MC method. To mitigate the so-called \textit{teleportation error} \cite{fleck1984random}, we extend the continuous source tilting technique from \cite{shi2020continuous} to FRTEs, preserving the correct equilibrium diffusion limit. Formal AP analysis in the diffusive scaling is established.
 Numerical experiments demonstrate that the proposed method is significantly more efficient than the IMC method \cite{fleck1971implicit}, particularly in optically thick regime.

The rest of the paper is organized as follows. In Section \ref{sec:model}, we revisit the model equations and derive the approximation model in the general multi-dimensional case. The numerical methods for the approximation model are described in Section \ref{sec:method}. Formal asymptotic analysis are given in Section \ref{sec:analysis}. We provide some numerical tests in Section \ref{sec:results}, followed by a conclusion remark in Section \ref{sec:conclusion}.

\section{Model approximation and reformulation}
\label{sec:model}
The frequency-dependent radiative transfer equations in the absence of material motion, scattering, heat conduction, and internal sources can be written as \cite{larsen2013properties}
\begin{subequations} \label{eq_RTE}
    \begin{align}
        \label{eq:simplified_transfer_intensity}
		&\frac{1}{c} \frac{\partial I}{\partial t}+\vOmega \cdot \nabla I=\sigma (B -I), 
             \\ 
		&\pdrv{u_m}{t}=\int_0^{\infty} \int_{4 \pi} \sigma(I-B) \dd{\vOmega}\dd{\nu},
    \end{align}
\end{subequations}
with the prescribed initial conditions
\begin{subequations}
    \begin{align}
        &I(\vx,\vOmega,\nu,0) = I^{i}(\vx,\vOmega,\nu),
        \\
        &T(\vx,0) = T^{i}(\vx),
    \end{align}
\end{subequations}
and the prescribed inflow boundary conditions
\begin{equation}
    I(\vx,\vOmega,\nu,t)\big|_{\vx \in \partial V} = I_{bc}(\vx,\vOmega,\nu,t),\quad \vOmega\cdot \mathbf{n} < 0.
\end{equation}
This system describes the radiative transfer and energy exchange between the radiation and the material. Here, $\vx \in V \subset \mathbb{R}^{3}$ is the spatial variable on a specified physical domain $V$,  
$\vOmega \in \mathbb{S}^2$ is the angular variable on the unit sphere $\mathbb{S}^2$,  
$\nu \in \mathbb{R}^{+}$ is the frequency variable,  
$t \in \mathbb{R}^{+}$ is the temporal variable,  
$\mathbf{n}$ is the unit outward normal vector on $\partial V$,  
$u_m(T)$ is the material energy density,  
$\sigma(\vx, \nu, T)$ is the opacity of the material, and  
$c$ is the speed of light.
  The two main unknowns are:
\begin{equation*}
    \begin{aligned}
        &I(\vx,\vOmega,\nu,t) =\text { the radiation intensity, }
        \\
        & T(\vx,t) = \text{the material temperature}.
    \end{aligned}
\end{equation*}
The Planck function $B(\nu,T) $ is defined as
    \begin{equation}\label{eq_def_PlanckFunc}
		B(\nu, T)=\frac{2 h \nu^3}{c^2} \frac{1}{e^{h \nu / k T}-1},
    \end{equation}
with the Boltzmann’s constant $k$  and the Planck’s constant $h$. The function $B(\nu,T) $ satisfies
\begin{equation}
    \int_0^{\infty} B(\nu,T) \dd{\nu} = \frac{1}{4\pi}acT^4,
\end{equation}
where $a = \frac{8 \pi^5 k^4}{15h^3c^3}$ is the radiation constant. The normalized Planck function $b(\nu,T)$
is defined by 
    \begin{equation}\label{eq_def_b}
		\begin{aligned}
			b(\nu, T)  =\frac{B(\nu, T)}{\int_0^{\infty} B(\nu, T) \dd{\nu}} = \frac{4 \pi}{ac T^4}B(\nu, T) ,
		\end{aligned}
    \end{equation}
thus we have $\int_0^{\infty} b(\nu,T) \dd{\nu} = 1$. The material energy density $u_m $ is related to the material temperature $T$ through the following equation of state
\begin{equation*}
    \pdrv{u_m}{T} = C_v(\vx,T),
\end{equation*}
where $C_v$ is the heat capacity of material. For simplicity, in this work, it is assumed  that $C_v$ is independent of the material temperature $T$, and we use the relation 
\begin{equation*}
    \pdrv{u_m}{t} = C_v \pdrv{T}{t},
\end{equation*}
in our subsequent discussion.

The spatial variable $\vx$ is usually presented by the Cartesian coordinate with $\vx = (x,y,z)$, while the angular variable $\vOmega $ is given by the spherical coordinates $(\theta,\varphi)$ with polar angle $\theta\in[0,\pi]$ and azimuthal angle $\varphi\in[0,2\pi]$, then
\begin{equation*}
\vOmega = (\xi,\eta,\mu), \quad \mu= - \cos \theta, \quad \xi=\sin \theta \cos \varphi, \quad \eta=\sin \theta \sin \varphi,
\end{equation*}
and
\begin{equation*}
    \dd{\vx} = \dd{x}\dd{y}\dd{z} \qc \dd{\vOmega} = \sin\theta\dd{\theta}\dd{\varphi} = \dd{\mu}\dd{\varphi}.
\end{equation*}
In the one-dimensional (1D) case, \eqref{eq_RTE} reduces to
\begin{equation} \label{eq_1DRTE}
    \begin{aligned}
		&\frac{1}{c} \frac{\partial I}{\partial t}+\mu \pdrv{I}{x}=\sigma ( B -I), 
             \\ 
		&C_v \pdrv{T}{t}= 2 \pi \int_0^{\infty}  \int_{-1}^{1} \sigma(I- B) \dd{\mu}  \dd{\nu},
    \end{aligned}
\end{equation}
with $\mu\in[-1,1]$, note that here we use $x$ to replace $z$ for convenience. While in the two-dimensional (2D) case, it becomes
\begin{equation} \label{eq_2DRTE}
    \begin{aligned}
		&\frac{1}{c} \frac{\partial I}{\partial t}+\xi \pdrv{I}{x} + \eta \pdrv{I}{y}=\sigma (B -I), 
             \\ 
		&C_v \pdrv{T}{t}= \int_0^{\infty} \int_{0}^{2 \pi} \int_{-1}^{1}  \sigma(I-B) \dd{\mu} \dd{\varphi } \dd{\nu},
    \end{aligned}
\end{equation}
with
\begin{equation*}
\xi=\sqrt{1-\mu^2} \cos \varphi \in[-1,1], \quad \eta=\sqrt{1-\mu^2} \sin \varphi \in[-1,1], \quad \mu \in[-1,1], \quad \varphi \in[0,2 \pi] .
\end{equation*}
\subsection{The Multi-group method}
Using the multi-group method, the continuous frequency space  $(0,\infty)$ is divided into $G$ groups, where all photons within a given group are treated with a single representative frequency, assigning an averaged opacity. The frequency interval is denoted by $(\nu_{g-\frac{1}{2}}, \nu_{g+\frac{1}{2}})$ for $g = 1,\dots,G$, with $\nu_{\frac{1}{2}}=0$ and $\nu_{G+\frac{1}{2}}=\infty$. In practice, a cutoff frequency of $\nu_{\frac{1}{2}}$ and $\nu_{G+\frac{1} {2}}$ is usually taken, which will be specified in the numerical tests. For each $g=1,\ldots,G$, we define the group-wise radiation intensity $I_g$ and Planck function $B_g$ as the integral over the frequency interval $(\nu_{g-\frac{1}{2}}, \nu_{g+\frac{1}{2}})$:
 \begin{subequations}
     \begin{align}
         &I_g(\vx, \vOmega, t) := \int_{\nu_{g-\half}}^{\nu_{g+\half}}I(\vx, \vOmega, \nu, t)\dd{\nu}, \\
         \label{eq_def_B_g}
         &B_g(T) := \int_{\nu_{g-\half}}^{\nu_{g+\half}} B(\nu,T) \dd{\nu}.
     \end{align}
 \end{subequations}
 With this definition, we can rewrite \eqref{eq_RTE}  as,
     \begin{subequations}\label{eq_microscopicSystem}
		\begin{align}
            &\begin{aligned}\label{eq_microscopicSystem_A}
                \frac{1}{c} \frac{\partial I_g}{\partial t}+\vOmega \cdot \nabla I_g = \sigma_{g}(B_g-I_g),  \quad g=1,\ldots,G, 
            \end{aligned}
            \\
		  &\begin{aligned}\label{eq_microscopicSystem_B}
                C_v \pdrv{T}{t}=\sum_{g=1}^{G} \int_{4 \pi} \sigma_g(I_g-B_g) \dd{\vOmega},
            \end{aligned}
		\end{align}
   \end{subequations} 
where a piecewise constant approximation is adopted for the group-wise opacity $\sigma_g$:
 \begin{equation}\label{eq_Approximation_sigma}
     \sigma_{g}(\vx,T) :=\frac{\int_{\nu_{g-\half}}^{\nu_{g+\half}} \sigma(\vx,\nu,T) (B-I)  \dd{\nu}}{\int_{\nu_{g-\half}}^{\nu_{g+\half}} (B-I) \dd{\nu}}\approx \frac{\int_{\nu_{g-\half}}^{\nu_{g+\half}} \sigma(\vx,\nu,T)  \dd{\nu}}{\nu_{g+\half}-\nu_{g-\half}}.
 \end{equation}
 For a detailed discussion on the approximation for $\sigma_g$, we refer readers to \cite{ pomraning2005equations,ZHANG2023112368}. The corresponding group-wise initial and boundary condition are given by:
   \begin{equation} 
        \begin{aligned}
               &I_{g}(\vx, \vOmega, 0) = I_{g}^i(\vx, \vOmega) :=  \int_{\nu_{g-\half}}^{\nu_{g+\half}} I^{i}(\vx,\vOmega,\nu) \dd{\nu}, \\[8pt]
               &I_{g}(\vx, \vOmega, t)\big|_{\vx \in \partial V} = I_{bc,g}(\vx, \vOmega, t):=  \int_{\nu_{g-\half}}^{\nu_{g+\half}}  I_{bc}(\vx,\vOmega,\nu,t) \dd{\nu}, \quad \vOmega \cdot \mathbf{n} < 0.
        \end{aligned}
    \end{equation}

\subsection{A semi-Lagrangian approximation} \label{Sec_SL}
Now, we  operate on  \eqref{eq_microscopicSystem_A} with  $\int_{4 \pi} (\cdot) \dd{\vOmega}$, the  system for macroscopic variables can be expressed as:
    \begin{subequations} \label{eq_macroscopic_system}
		\begin{align}
        \label{eq_macroscopic_system_A}
			&\frac{1}{c} \frac{\partial \rho_g}{\partial t}+\nabla \cdot \int_{4 \pi} \vOmega I_g \, \dd{\vOmega} = \sigma_{g} ( 4 \pi B_g-\rho_g ) , \quad  g=1, \ldots , G, \\  
			&C_v \pdrv{T}{t}=\sum_{g=1}^{G} \sigma_{g} \left( \rho_g- 4 \pi B_g \right),
		\end{align}
    \end{subequations} 
    where $\rho_g := \int_{4 \pi} I_g \dd{\vOmega} $.
    
The primary objective of this subsection is to utilize the  integral form of the microscopic equation \eqref{eq_microscopicSystem_A} with suitable approximation to close the flux $\int_{4 \pi} \vOmega I_g \, \mathrm{d}\vOmega$ in the macroscopic equation \eqref{eq_macroscopic_system_A}. This idea originally comes from the UGK scheme \cite{xu2010unified,mieussens2013asymptotic}. Whereas UGK scheme reconstructs the emission source term using a local linear polynomial derived from neighboring cells, our approach directly obtains the macroscopic flux through integration by parts of the time integral, combined with a semi-Lagrangian approximation \cite{shi2023efficient,zhang2023asymptotic,cai2024asymptotic}. This approach enables the use of a larger time step, which can be independent of the speed of light.


For each $g=1,\ldots,G$, we rewrite the microscopic equation \eqref{eq_microscopicSystem_A} into the following characteristic form
    \begin{subequations}
        \begin{align}\label{eq_characteristic_form}
            \frac{\dd{I_g}}{\dd{t}} &= c \sigma_g ( B_g - I_g),
            \\
            \frac{\dd{\vx}}{\dd{t}} &= c \vOmega,
        \end{align}
    \end{subequations}
with $\dd{} / \dd{t} $ being a material derivative.  
Without loss of generality, we formally impose a time interval $[t^n,t^{n+1}]$.
 Starting from $(\vx,t)$ with $t \in [t^n,t^{n+1}]$, a backward tracing of the characteristic line is given by:
    \begin{equation}\label{eq_def_characteristic_line}
        \mfx(s,\vOmega;\vx,t)=\vx-c \vOmega (t-s), \quad \forall \, s \in [t^n,t].
    \end{equation}
For simplicity, we denote $\mfx(s)$ as shorthand for $\mfx(s,\vOmega;\vx,t)$.  By multiplying \eqref{eq_characteristic_form} with an exponential factor $e^{c\sigma_g s  }$ and integrating over $s \in [t^n,t]$, we obtain:
    \begin{equation}
		\begin{aligned}\label{eq_formal_solution}
			I_g( \vx, \vOmega,t)= & e^{- c  \sigma_g (t - t^n)} I_{g}( \mfx(t^n), \vOmega, t^n ) 
            \\ 
			& +\int_{t^n}^{t} e^{- c  \sigma_g (t - s )} c\sigma_{g} B_g ( \mfx(s),  s )\dd{s}.
		\end{aligned}
    \end{equation}
Integrating by parts for the second term in \eqref{eq_formal_solution}, we get
    \begin{equation}\label{eq_formal_solution_integrated}
        \begin{aligned}
           \int_{t^n}^{t} e^{- c  \sigma_g (t - s )}  c\sigma_{g} B_g ( \mfx(s),  s )\dd{s}
           =   \Bigg (  B_g(\vx,t) &  - e^{- c  \sigma_g (t - t^n )}
            B_g( \mfx(t^n),t^n )
            \\
            & - \int_{t^n}^t e^{- c  \sigma_g (t - s )} \frac{\dd{B_g}}{\dd{s }}( \mfx(s ),s ) \dd{s } \Bigg ),
        \end{aligned}
    \end{equation}
    where the material derivative is given by
    \begin{equation*}
        \frac{\dd{B_g}}{\dd{s }}( \mfx(s ),s ) = \pdrv{B_g}{s}( \mfx(s ),s )  + c \vOmega \cdot \nabla B_g ( \mfx(s ),s )  .
    \end{equation*}
For the time integral in \eqref{eq_formal_solution_integrated},  we follow the idea in \cite{shi2023efficient,zhang2023asymptotic,cai2024asymptotic} to  give the following approximation,
    \begin{equation} \label{eq_lastTerm2}
        \begin{aligned}
             \int_{t^n}^t e^{- c  \sigma_g (t - s )} \frac{\dd{B_g}}{\dd{s }}( \mfx(s ),s ) \dd{s } 
            &\approx     \int_{t^n}^t e^{- c \sigma_g (t - s  )     } \dd{s} \left( 
 \frac{\dd{B_g}}{\dd{s}}(\mfx(s),s)  \Bigg|_{s = t} \right) 
                        \\
            &= \frac{1}{c\sigma_g} (1 -  e^{- c \sigma_g (t - t^{n} ) } ) \left( \pdrv{B_g}{s}(\vx,t) +  c \vOmega \cdot \nabla B_g(\vx,t)  \right) .
        \end{aligned}
    \end{equation}
Together, we deduce
     \begin{equation}\label{eq_formal_solution_approximated}
		\begin{aligned}
		  I_g( \vx, \vOmega,t)
            \approx & \,e^{- c  \sigma_g (t - t^n )} I_{g}( \mfx(t^n), \vOmega, t^n ) 
            \\
            & + \Bigg( B_g(\vx,t)  - e^{- c  \sigma_g (t - t^n )}
            B_g( \mfx(t^n),t^n )  \\ 
			& \quad  -  \frac{1}{c\sigma_g} (1 -  e^{- c \sigma_g (t - t^{n} ) } )  \left( \pdrv{B_g}{s}(\vx,t) +  c \vOmega \cdot \nabla B_g(\vx,t)  \right)   \Bigg).
		\end{aligned}
    \end{equation}
     By substituting \eqref{eq_formal_solution_approximated} into the macroscopic equation \eqref{eq_macroscopic_system_A} to close the flux
$\int_{4\pi} \vOmega I_g \,\mathrm{d}\vOmega$, 
and using the definition of the characteristic line \eqref{eq_def_characteristic_line} (which indicates $\mfx(t^n)$ is angular-dependent), as well as the integrals 
$\int_{4\pi} \vOmega \cdot \mathbf{v} \,\mathrm{d}\vOmega  = 0$ and $\int_{4\pi} \vOmega (\vOmega \cdot \mathbf{v}) \,\mathrm{d}\vOmega = \frac{4\pi}{3}\mathbf{v} $ for any variable  $\mathbf{v} $ that is angular-independent, 
we obtain 
\begin{subequations} \label{eq_approximaed_nonModifed}
		\begin{align}
			\frac{1}{c} \frac{\partial \rho_g}{\partial t}&+\nabla \cdot\int_{4 \pi}  e^{- c  \sigma_g (t - t^n )} \vOmega   \left( I_{g}  - B_{g} \right) ( \mfx(t^n), \vOmega,t^n ) \dd{\vOmega} 
            \\ \nonumber
		  & - \nabla \cdot \left( \frac{4 \pi }{3\sigma_{g}}\left(1-e^{-c \sigma_{g}\left(t-t^n\right)}\right) \nabla B_g \right)=  \sigma_{g} ( 4 \pi B_g-\rho_g ),  \quad g=1,\ldots,G, 
            \\  
		  C_v \pdrv{T}{t}&=\sum_{g=1}^{G} \sigma_{g} \left( \rho_g- 4 \pi  B_g\right) .
		\end{align}
    \end{subequations} 
    
\begin{remark}
    The approximation error in \eqref{eq_lastTerm2} has been proved to be $O(\Delta t^2)$ for linear transport equation in \cite{zhang2023asymptotic,cai2024asymptotic}.
\end{remark}

\begin{remark}
Our proposed model differs from UGK-type models for FRTEs \cite{sun2015asymptoticB,li2024unified,yang2025unified}. The transport process is decomposed into two components: convection and diffusion. Specifically, the convection term involves backward tracing the microscopic perturbation $\left( I_g - B_g \right)$ along the characteristic line, while the diffusion term is obtained by approximating $\nabla B_g$ along the same characteristic line. This approach reveals that the approximation of the flux $\int_{4\pi} \vOmega I_g \,\mathrm{d}\vOmega$ is not solely a function of neighboring cells, thereby eliminating the time step constraint imposed by the explicit CFL condition in the following numerical schemes.

\end{remark}

    \subsection{Model correction for the semi-Lagrangian approximation}
    When the radiative transfer equation reduces to the gray case, the approximated macroscopic system becomes identical to the formulation presented in our previous work \cite{shi2023efficient}, whose effectiveness has been validated through numerical experiments.  However, it is important to note that the diffusion coefficient, given by $\frac{1}{3 \sigma_g}(1-e^{-c\sigma_g (t-t^n)})$, approaches $\frac{c(t-t^n)}{3}$ rather than $0$ as $\sigma_g \rightarrow 0$. This indicates that, unlike the decomposed HOLO scheme \cite{feng2025decomposed}, the proposed model does not formally recover the free-streaming limit. While this discrepancy is generally acceptable in purely free-streaming regime (where $\nabla B_g \approx 0$), it becomes problematic at optically thick-thin interface and because of the frequency-dependent nature of the problem. To address this limitation, we introduce a model correction for our original model in the following subsection, to make it applicable also in the free streaming regime.

 We introduce $\theta_g \in [0,1]$ to give a convex combination of the time integral in \eqref{eq_formal_solution}:
    \begin{equation}\label{eq_formal_solution2}
		\begin{aligned}
			I_g( \vx, \vOmega,t)= & e^{- c  \sigma_g (t - t^n )} I_{g}( \mfx(t^n), \vOmega, t^n ) 
            \\ 
			& +\theta_g \int_{t^n}^{t} e^{- c  \sigma_g (t - s )} c\sigma_{g} B_g ( \mfx(s),  s )\dd{s}
            \\
            & +(1-\theta_g)\int_{t^n}^{t} e^{- c  \sigma_g (t - s )} c\sigma_{g} B_g ( \mfx(s),  s )\dd{s},
		\end{aligned}
    \end{equation}
     where $\theta_g$ denotes a local varying weight function associated with the opacity $\sigma_g$, characterizing the optical thickness for each group $g$. 
To capture the asymptotic behavior, we aim to keep
\begin{itemize}
    \item $\theta_g \to 1$ in the free-streaming regime (relatively small $\sigma_g$),
    \item $\theta_g \to 0$ in the diffusive regime (relatively large $\sigma_g$).
\end{itemize}
We achieve this dual asymptotic requirement through the exponential definition $\theta_g(\sigma_g,t) = e^{-c \sigma_{g} (t-t^n)}$ or $\theta_g(\sigma_g,t)= 1 - e^{-1 / (c \sigma_{g} (t-t^n))}$. Numerical experiments indicate that the proposed scheme is not sensitive to the choice of the weight functions. See Section \ref{weightFunction_compared} for further details. In this work, we take $\theta_g(\sigma_g,t) = e^{-c \sigma_{g} (t-t^n)}$.

    For the $ \theta_g $ term in \eqref{eq_formal_solution2}, we assume that this part represents a less stiff source, which contributes minimally in the diffusive regime. As a result, the time integral is directly approximated by
 \begin{equation}
		\begin{aligned}\label{eq_approximation2}
      \theta_g \int_{t^n}^{t} e^{- c  \sigma_g (t - s )} c\sigma_{g} B_g ( \mfx(s),  s )\dd{s} 
        & \approx  \theta_g \int_{t^n}^t c \sigma_{g} e^{-c \sigma_{g}(t-s)}    \dd{s} \left(  B_g  ( \mfx(s),  s )  \Big|_{s = t} \right)
        \\ 
		& = \theta_g ( 1 - e^{-c \sigma_{g}(t-t^n)}) B_g    \left( \vx,  t\right).
		\end{aligned}
	\end{equation}
On the other hand, for the $ 1 - \theta_g $ term in equation \eqref{eq_formal_solution2}, we assume it corresponds to the stiff source, and we apply the approximation from the last subsection:
   \begin{equation}
        \begin{aligned}
           (1-\theta_g)\int_{t^n}^t e^{- c  \sigma_g (t - s )} c\sigma_{g} B_g ( \mfx(s),  s ) \dd{s}&\approx  (1-\theta_g)  \Bigg (  B_g(\vx,t) - e^{- c  \sigma_g (t - t^n )}  
            B_g( \mfx(t^n),t^n )
            \\
            & - \frac{1}{c\sigma_g} (1 -  e^{- c \sigma_g (t - t^{n} ) } ) \left( \pdrv{B_g}{s}(\vx,t) +  c \vOmega \cdot \nabla B_g(\vx,t)  \right) \Bigg).
        \end{aligned}
    \end{equation}
Together, we have 
    \begin{equation}\label{eq_formal_solution_approximated2}
		\begin{aligned}
		  I_g( \vx, \vOmega,t)
            \approx & \, e^{- c  \sigma_g (t - t^n)} I_{g}( \mathbf{r}(t^n), \vOmega,t^n) 
            \\ 
		  & +  \theta_g     ( 1 - e^{-c \sigma_{g}(t-t^n)}) B_g  \left( \vx,  t\right) 
            \\
            & + ( 1 - \theta_g  )  \Bigg( B_g \left( \vx,  t\right)-e^{- c  \sigma_g (t - t^n)} B_g \left( \mathbf{X}(t^n),  t^n\right)  \\ 
			& \quad  -  \frac{1}{c \sigma_g} (1 - e^{-c \sigma_{g}(t-t^n)}) \left( \pdrv{B_g}{s}(\vx,t) +  c \vOmega \cdot \nabla B_g(\vx,t)  \right)   \Bigg).
		\end{aligned}
    \end{equation}
    Similarly,  substituting \eqref{eq_formal_solution_approximated2} into the macroscopic equation \eqref{eq_macroscopic_system_A} to close the flux
$\int_{4\pi} \vOmega I_g \,\mathrm{d}\vOmega$, it yields
    \begin{subequations}    \label{eq_macroscopic_system_approximated}
		\begin{align}
			\frac{1}{c} \frac{\partial \rho_g}{\partial t}&+\nabla \cdot\int_{4 \pi}  e^{- c  \sigma_g (t - t^n)} \vOmega   I_{g} ( \mfx(t^n), \vOmega,t^n ) \dd{\vOmega} 
            \\ \nonumber
            & - \nabla \cdot \left( ( 1 - \theta_g  ) \int_{4 \pi}  e^{- c  \sigma_g (t - t^n)} \vOmega  B_g  \left( \mfx(t^n),t^n \right) \dd{\vOmega} \right)
            \\ \nonumber
		  & - \nabla \cdot \left( ( 1 - \theta_g  ) \frac{4 \pi }{3\sigma_{g}}\left(1-e^{-c \sigma_{g}\left(t-t^n\right)}\right) \nabla B_g \right)=  \sigma_{g} ( 4 \pi B_g-\rho_g ),  \quad g=1,\ldots,G, 
            \\  
		  C_v \pdrv{T}{t}&=\sum_{g=1}^{G} \sigma_{g} \left( \rho_g- 4 \pi  B_g\right) .
		\end{align}
    \end{subequations} 

    \subsection{Reformulation}
    
We introduce the following notations to reformulate \eqref{eq_macroscopic_system_approximated}:
    \begin{subequations}
          \begin{align}
              &b_g(T) := \int_{\nu_{g-\half}}^{\nu_{g+\half}} b(\nu,T) \dd{\nu},
              \\
              &\phi(T)  := acT^4.
          \end{align}
    \end{subequations}
From the definition of the normalized Planck function in \eqref{eq_def_b}, together with the definition of the group-wise Planck function in \eqref{eq_def_B_g}, we derive the following relationship:
\begin{equation} \label{eq_def_phig}
        4 \pi B_g = b_g \phi.
\end{equation}
Applying the chain rule yields
\begin{equation}\label{eq_def_dBdT}
    \begin{aligned}
            4 \pi \nabla B_g 
            &=4 \pi \pdrv{B_g}{T} \nabla T 
            \\
            &=(b_g + \frac{T}{4} \pdrv{b_g}{T})4acT^3 \nabla T 
            \\
            &=(b_g + \frac{T}{4} \pdrv{b_g}{T}) \nabla \phi.
    \end{aligned}
\end{equation}
Therefore, we can reformulate the macroscopic system \eqref{eq_macroscopic_system_approximated} as:
    \begin{subequations}
    \label{eq_macroscopic_system_approximated_expressed}
		\begin{align}
			\frac{1}{c} \frac{\partial \rho_g}{\partial t}&+\nabla \cdot\int_{4 \pi}  e^{- c  \sigma_g (t - t^n)} \vOmega   I_{g} ( \mfx(t^n), \vOmega,t^n ) \dd{\vOmega} 
            \\ \nonumber
            & - \nabla \cdot \left ( ( 1 - \theta_g  ) \int_{4 \pi}  e^{- c  \sigma_g (t - t^n)} \vOmega  B_g \left( \mfx(t^n),t^n \right) \dd{\vOmega} \right )
            \\ \nonumber \label{eq_energyBalance}
		  & - \nabla \cdot \left ( ( 1 - \theta_g  ) \frac{1 }{3\sigma_{g}}\left(1-e^{-c \sigma_{g}\left(t-t^n\right)}\right) (b_g + \frac{T}{4} \pdrv{b_g}{T}) \nabla \phi \right ) =  \sigma_{g} ( b_g  \phi - \rho_g ),  \quad g=1,\ldots,G, 
            \\  
		  C_v \pdrv{T}{t}&=\sum_{g=1}^{G} \sigma_{g} \left( \rho_g- b_g  \phi \right) .
		\end{align}
    \end{subequations} 

    \begin{remark} \label{rk_reformulation}
        We use the notation $b_g \phi$ and $(b_g + \frac{T}{4} \pdrv{b_g}{T}) \nabla \phi$ to replace $4 \pi B_g$ and $4 \pi \nabla B_g$, respectively. This formulation offers an advantage: it enables our numerical implementation to seamlessly transition between frequency-dependent and gray radiative transfer equations. To verify this, observing that in the gray case $G =1$, the coefficients simplify to be $b_g = 1$ and $(b_g + \frac{T}{4} \pdrv{b_g}{T}) = 1$ for all $g=1,\ldots,G$. 
    \end{remark}

\section{Numerical method}
\label{sec:method}
In this section, we present the numerical method for solving the coupled macro-micro system given by equations \eqref{eq_microscopicSystem} and \eqref{eq_macroscopic_system_approximated_expressed}. We begin by discretizing \eqref{eq_macroscopic_system_approximated_expressed} using a hybrid finite volume method, where the convective flux  is provided by a MC method and the diffusive flux discretized implicitly with central difference. To handle the macroscopic system's nonlinearity that couples space and frequency dimensions, we employ a Picard iteration combined with a predictor-corrector approach, which decouples the system into space-only linear equations and cell-local scalar nonlinear equations.  The resulting material temperature $T^{n+1}$ provides  a \textit{priori} estimate for the emission source $B_g^{n+1}$, thereby reducing \eqref{eq_microscopicSystem_A} to a purely absorbing radiative transport problem with a known source. This problem can be efficiently solved using a MC method.

We partition the computational domain $V$ into $N_x$ cells $\{V_i\}_{i=1}^{N_x}$. 
Let $\partial V_i$ denote the boundary of the cell $V_i$, and let $S_{ij} = \partial V_i \cap \partial V_j$ denote the interface shared between the neighboring cells $V_i$ and $V_j$. 
We denote by $\Delta V_i$ and $|S_{ij} | $ the volume of $V_i$ and the area of $S_{ij}$, respectively. 
The set of indices of the neighboring cells of $V_i$ that share a face is denoted by $\mathcal{N}_i$. 
Let $\mathbf{n}_{ij}$ be the unit normal vector on $S_{ij}$, oriented from $V_i$ toward $V_j$. 
Finally, let $\mathbf{x}_i$ denote the barycenter (center of mass) of the cell $V_i$.

Considering a time interval $[0, T]$, we define the time step size $\Delta t = T / N_t$, where $N_t$ is a positive integer. 
We then set the discrete time levels by $t^n = n \Delta t$, for $n = 1, \ldots, N_t$.

\subsection{A finite volume method for the macroscopic system}
We integrate equation \eqref{eq_macroscopic_system_approximated_expressed} over a time interval $[t^n,t^{n+1}]$ and a space cell $V_i$ to obtain the following finite volume scheme:
\begin{subequations} \label{eq_fullDiscretization} 
		\begin{align} 
			&\frac{\rho_{g,i}^{n+1} - \rho_{g,i}^{n}}{c \Delta t} +  \frac{1}{\Delta V_i}  \sum_{j \in \mathcal{N}_i} \left( F_{g,ij}^{C,n+1} - F_{g,ij}^{D,n+1} \right) =  \sigma_{g,i}^{n+1}b_{g,i}^{n+1}  \phi_i^{n+1} -\sigma_{g,i}^{n+1}\rho_{g,i}^{n+1},  \quad g=1,\ldots,G, \label{eq_fullDiscretization_A} 
            \\  
		  &C_{v,i} \frac{T^{n+1}_i - T^n_i}{\Delta t}=\sum_{g=1}^{G} \sigma_{g,i}^{n+1} \left( \rho_{g,i}^{n+1}- b_{g,i}^{n+1}  \phi_i^{n+1} \right) .\label{eq_fullDiscretization_B} 
		\end{align}
    \end{subequations} 

The  convective flux $F_{g,ij}^{C,n+1}$ is defined as:
\begin{equation}          \label{eq_fulldis_fluxA}
        \begin{aligned}  
            &F_{g,ij}^{C,n+1} =  
            \Bar{F}_{g,ij}^{I,n+1} 
            - (1 - \theta_{g,i}^{n+1}) \Bar{F}_{g,ij}^{B,n+1,+}
            - (1 - \theta_{g,j}^{n+1}) \Bar{F}_{g,ij}^{B,n+1,-},
        \end{aligned}
    \end{equation}
     where the weight function  $\theta_{g,i}^{n+1}$  is  defined implicitly as 
    \begin{equation}
        \theta_{g,i}^{n+1} =  e^{-c \sigma_{g,i}^{n+1} \Delta t}. 
    \end{equation}
    The component fluxes are defined as follows:
    \begin{itemize}
        \item The term $\Bar{F}_{g,ij}^{I,n+1}$ denotes the surface flux contribution from the initial and boundary sources \eqref{eq_sample_I} and \eqref{eq_sample_bc}, defined as
        \begin{equation}
            \Bar{F}_{g,ij}^{I,n+1} = \frac{1}{\Delta t} 
            \int^{t^{n+1}}_{t^n}  
            \int_{S_{ij} }  
            \int_{4 \pi}  
           e^{- c  \sigma_g (t - t^n)}   
            I_{g} ( \mfx(t^n), \vOmega,t^n ) 
            \vOmega \cdot \mathbf{n}_{ij} 
            \dd{\vOmega} \dd{S}  \dd{t}.
        \end{equation}
            \item The term $\Bar{F}_{g,ij}^{B,n+1,+}$  denotes the outflow surface flux contribution from the ghost initial and boundary sources \eqref{eq_sample_B} and \eqref{eq_sample_B_bc}, defined as
            \begin{equation}
            \Bar{F}_{g,ij}^{B,n+1,+}= 
            \frac{1}{\Delta t} 
            \int^{t^{n+1}}_{t^n}   
            \int_{S_{ij} }  
            \int_{\vOmega \cdot \mathbf{n}_{ij} > 0 }
            e^{- c  \sigma_g (t - t^n)}  B_{g} ( \mfx(t^n),t^n )  \vOmega \cdot \mathbf{n}_{ij} 
            \dd{\vOmega} \dd{S} \dd{t}.
        \end{equation}
            \item The term $\Bar{F}_{g,ij}^{B,n+1,-}$ denotes the inflow surface flux contribution from the ghost initial and boundary sources \eqref{eq_sample_B} and \eqref{eq_sample_B_bc}, defined as
                        \begin{equation}
            \Bar{F}_{g,ij}^{B,n+1,-}= 
            \frac{1}{\Delta t} 
            \int^{t^{n+1}}_{t^n}   
            \int_{S_{ij} }  
            \int_{\vOmega \cdot \mathbf{n}_{ij} < 0  }
            e^{- c  \sigma_g (t - t^n)}  B_{g} ( \mfx(t^n),t^n )  \vOmega \cdot \mathbf{n}_{ij} 
            \dd{\vOmega} \dd{S} \dd{t}.
        \end{equation}
    \end{itemize}
    All these integrals are evaluated using a MC method, with details provided in \eqref{eq_def_fluxA} and \eqref{eq_def_fluxB}. 

    The diffusive flux $F^{D,n+1}_{g,ij}$ is discretized as:
    \begin{equation} \label{eq_def_diffusiveFLux}
        F^{D,n+1}_{g,ij} = D_{g,ij}^{n+1}\frac{\phi_{j}^{n+1} - \phi_i^{n+1}}{ |\vr_j - \vr_i| }|S_{ij}|,
    \end{equation}
    where the diffusion coefficient $D_{g,ij}^{n+1}$ is given by:
    \begin{equation}          \label{eq_fulldis_fluxB}
        \begin{aligned}  
            D_{g,ij}^{n+1} =&  ( 1 - \theta_{g,ij}^{n+1}  ) \frac{1 }{3\sigma_{g,ij}^{n+1}}(1-e^{-c \sigma_{g,ij}^{n+1}\Delta t}) (b_g + \frac{T}{4} \pdrv{b_g}{T})^{n+1}_{ij}.
        \end{aligned}
    \end{equation}
The interface value $\theta_{g,ij}^{n+1}$ is approximated by the arithmetic average:
\begin{equation}
    \theta_{g,ij}^{n+1} = \half(\theta_{g,i}^{n+1} + \theta_{g,j}^{n+1}   ).
\end{equation}
 The interface value $\sigma_{g,ij}$ is evaluated using the harmonic average:
    \begin{equation}\label{eq_eva_sigma} 
        \sigma_{g,ij}^{n+1} = \frac{2 \sigma_{g,i}^{n+1} \sigma_{g,j}^{n+1}}{\sigma_{g,i}^{n+1}  + \sigma_{g,j}^{n+1} }.
    \end{equation}
Finally, the quantity $(b_g + \frac{T}{4} \pdrv{b_g}{T})^{n+1}_{ij}$ is computed  based on the interface temperature 
 $T^{n+1}_{ij}$, which is determined by
    \begin{equation} \label{eq_eva_T}
        T^{n+1}_{ij} = \left(\frac{(T^{n+1}_{i})^4 + (T^{n+1}_{j})^4}{2}  \right)^{\frac{1}{4}} .
    \end{equation}
    \begin{remark}
        For nonuniform cells, it is recommended to employ cell size-based weighting in the evaluation of \eqref{eq_eva_sigma} and \eqref{eq_eva_T}. This strategy is used in the numerical tests.
    \end{remark}

    \begin{remark}
     We use a MC method to compute the convective flux, which, unlike the explicit deterministic methods, is not subject to the CFL time step constraint.  In addition, the diffusion term is derived by approximating $\nabla B_g$ along the characteristic line and is discretized implicitly, thereby avoiding both the CFL and the parabolic time step restriction.  These features reflect the long-range characteristic tracing inherent in our model,  allowing for the use of a large time step that is independent of the speed of light.
    \end{remark}
 \subsubsection{Picard iteration with a predictor-corrector procedure}
     Directly solving the scheme \eqref{eq_fullDiscretization} would require handling a fully coupled nonlinear system, where the coupling spans both spatial and frequency dimensions, leading to prohibitively high computational costs.  To circumvent this complexity, we adopt a Picard iteration in combination with a \textit{predictor-corrector} procedure, following the approach proposed in \cite{xiong2022high,tang2021accurate}. 
     
     To start with, we reformulate the system \eqref{eq_fullDiscretization}. From \eqref{eq_fullDiscretization_A}, we have 
     \begin{equation} \label{eq_rho_g}
     	\rho_{g,i}^{n+1}=\frac{\frac{1}{c \Delta t}\rho_{g,i}^n+ \sigma_{g,i}^{n+1} b_{g,i}^{n+1} \phi_i^{n+1} -   \frac{1}{\Delta V_i} \sum_{j \in \mathcal{N}_i} 
     		\left( F_{g,ij}^{C,n+1} -  F_{g,ij}^{D,n+1} \right)}{\frac{1}{c \Delta t}+ \sigma_{g,i}^{n+1}}.
     \end{equation}
     By substituting \eqref{eq_rho_g} into the equation  \eqref{eq_fullDiscretization_B} to eliminate $\rho_{g,i}^{n+1}$, we obtain
     \begin{equation} \label{eq_correction_A}
     	C_{v,i} \frac{T^{n+1}_i - T^n_i}{\Delta t}=\sum_{g=1}^{G} \chi_{g,i}^{n+1} \left( \frac{1}{c \Delta t}
     	\left(\rho_{g,i}^n -  b_{g,i}^{n+1} \phi_i^{n+1}  \right)  -   \frac{1}{\Delta V_i} \sum_{j \in \mathcal{N}_i} 
     	\left( F_{g,ij}^{C,n+1} -  F_{g,ij}^{D,n+1} \right)\right),
     \end{equation}
     where $\chi_{g,i}^{n+1} := \frac{\sigma_{g,i}^{n+1}}{\frac{1}{c \Delta t} + \sigma_{g,i}^{n+1}}$. Multiplying  \eqref{eq_energyBalance} by $4acT^3$ yields
     \begin{equation*}
     	C_v \pdrv{\phi}{t}=4acT^3\sum_{g=1}^{G} \sigma_{g}  \left(\rho_g- b_g  \phi \right) .
     \end{equation*} 
     Let $\beta := \frac{4acT^3}{C_v}$, this equation can be discretized implicitly as
     \begin{equation}\label{eq_energeyBalance_linear}
     	\begin{aligned}
     		\frac{ \phi_{i}^{n+1} - \phi_{i}^n }{\beta_{i}^{n+1} \Delta t} 
     		&=  \sum_{g=1}^{G} \sigma_{g,i}^{n+1} \left(  \rho_{g,i}^{n+1} 
     		-  b_{g,i}^{n+1} \phi_i^{n+1} \right).
     	\end{aligned}
     \end{equation} 
          By substituting \eqref{eq_rho_g} into the equation  \eqref{eq_energeyBalance_linear} to eliminate $\rho_{g,i}^{n+1}$, we obtain
        \begin{equation} \label{eq_prediction_A}
     	 \frac{\phi^{n+1}_i - \phi^n_i}{ \beta_i^{n+1} \Delta t}=\sum_{g=1}^{G} \chi_{g,i}^{n+1} \left( \frac{1}{c \Delta t}\left(\rho_{g,i}^n -  b_{g,i}^{n+1} \phi_i^{n+1}  \right)  -   \frac{1}{\Delta V_i} \sum_{j \in \mathcal{N}_i} 
     	\left( F_{g,ij}^{C,n+1} -  F_{g,ij}^{D,n+1} \right)\right) .
     \end{equation}

    Now the Picard iteration for solving \eqref{eq_fullDiscretization}  is defined as follows: for the iterative number $k$ starting at $k =0$, where $b_g^{n+1,0} $ , $\sigma_g^{n+1,0} $, $\chi_g^{n+1,0} $ and $D_g^{n+1,0} $ can be computed by taking the value of $T^{n+1,0} = T^n$, we update the unknowns $T^{n+1,k+1}$ iteratively by the following step:

    	 \textbf{The prediction step.} 
    	Using \eqref{eq_prediction_A}, we first solve the following  \textit{linear} system:
    	\begin{equation} \label{eq_prediction_B}
    		\frac{\phi^{n+1,k+\half}_i - \phi^n_i}{ \beta_i^{n+1,k} \Delta t}=\sum_{g=1}^{G} \chi_{g,i}^{n+1,k} \left( \frac{1}{c \Delta t} \left(\rho_{g,i}^n -  b_{g,i}^{n+1,k} \phi_i^{n+1,k+\half}  \right)
    		-   \frac{1}{\Delta V_i} \sum_{j \in \mathcal{N}_i} 
    		F_{g,ij}^{n+1,k+\half} \right) ,
    	\end{equation}
    	where 
    	\begin{equation*}
    		F_{g,ij}^{n+1,k+\half} = F_{g,ij}^{C,n+1,k} -  F_{g,ij}^{D,n+1,k+\half},
    	\end{equation*}
    	and
    	\begin{subequations}
    		\begin{align*}
    			& F_{g,ij}^{C,n+1,k} = \Bar{F}_{g,ij}^{I,n+1} 
    			- (1 - \theta_{g,i}^{n+1,k}) \Bar{F}_{g,ij}^{B,n+1,+}
    			- (1 - \theta_{g,j}^{n+1,k}) \Bar{F}_{g,ij}^{B,n+1,-},
    			\\
    			&F_{g,ij}^{D,n+1,k+\half} = D_{g,ij}^{n+1,k} \frac{\phi_{j}^{n+1,k+\half} - \phi_i^{n+1,k+\half}}{ |\vr_j - \vr_i| }|S_{ij}|,
    			\\
    			&D_{g,ij}^{n+1,k} =  ( 1 - \theta_{g,ij}^{n+1,k}  ) \frac{1 }{3\sigma_{g,ij}^{n+1,k}}(1-e^{-c \sigma_{g,ij}^{n+1,k}\Delta t}) (b_g + \frac{T}{4} \pdrv{b_g}{T})_{ij}^{n+1,k}.
    		\end{align*}
    	\end{subequations}
    	Notice that $b_{g,i}^{n+1,k}$ and $\left(b_g + \frac{T}{4} \frac{\partial b_g}{\partial T}\right)_{ij}^{n+1,k}$ are evaluated using the $k$th iteration values. This approach eliminates the need to solve a linear system that couples space and frequency dimensions. 
    	
    	To see how this forms a \textit{linear} system, we reformulate \eqref{eq_prediction_B} as follows:
    	\begin{equation}\label{eq_prediction}
    		\begin{aligned}
    			\Big(\frac{1}{\beta_i^{n+1,k} \Delta t} 
    			&+ \frac{1}{c \Delta t} \sum_{g=1}^G \chi_{g,i}^{n+1, k} b_{g,i}^{n+1, k}    \Big)\phi_i^{n+1,k+\frac{1}{2}}
    			\\
    			&-   \frac{1}{\Delta V_i} \sum_{g=1}^G \chi_{g,i}^{n+1, k} \left(  \sum_{j \in \mathcal{N}_i}   D_{g,ij}^{n+1,k}\frac{\phi_{j}^{n+1,k+\half} - \phi_i^{n+1,k+\half}}{ |\vr_j - \vr_i| }|S_{ij}|    \right)= \text{RHS}_i,
    		\end{aligned}
    	\end{equation}
    	where 
    	\begin{equation*}
    		\text{RHS}_i = \frac{1}{\beta_i^{n+1,k} \Delta t}\phi_i^n +  
    		\sum_{g=1}^G \chi_{g,i}^{n+1, k}  \left(\frac{1}{c \Delta t}\rho_{g,i}^n- \frac{1}{\Delta V_i} \sum_{j \in \mathcal{N}_i} F_{g,ij}^{C,n+1,k} \right).
    	\end{equation*}
    	Then at each iterative step, given $T_i^{n+1,k}$, equation \eqref{eq_prediction}  is a \textit{space-only linear} system with respect to $\phi_i^{n+1,k+\frac{1}{2}}$. A standard linear solver used for solving the Poisson equation can be applied to this linear system.
    	
    	\textbf{The correction step.}
    	With $\phi_i^{n+1,k+\frac{1}{2}}$ obtained from the prediction step, we compute the corresponding temperature:
    	\begin{equation*}
    		T_i^{n+1,k+\frac{1}{2}} = \left( \frac{\phi_i^{n+1,k+\frac{1}{2}}}{ac} \right)^{\frac{1}{4}},
    	\end{equation*}
    	which allows us to update the temperature-dependent terms $\sigma_{g,i}^{n+1,k+\frac{1}{2}}$, $\chi_{g,i}^{n+1,k+\frac{1}{2}}$, and $D_{g,ij}^{n+1,k+\frac{1}{2}}$. Using \eqref{eq_correction_A}, we solve the following \textit{scalar} nonlinear system:
    	\begin{equation} \label{eq_correction_B}
    		C_{v,i} \frac{T^{n+1,k+1}_i - T^n_i}{\Delta t}=\sum_{g=1}^{G} \chi_{g,i}^{n+1,k+\half} \left( \frac{1}{c \Delta t}\left(\rho_{g,i}^n -  b_{g,i}^{n+1,k+1} \phi_i^{n+1,k+1}  \right)  -   \frac{1}{\Delta V_i} \sum_{j \in \mathcal{N}_i} 
    		F_{g,ij}^{n+1,k+\half*} \right) .
    	\end{equation}
    	where 
    	\begin{equation*}
    		F_{g,ij}^{n+1,k+\half*} = F_{g,ij}^{C,n+1,k+\half} -  F_{g,ij}^{D,n+1,k+\half*},
    	\end{equation*}
    	and
    	\begin{subequations}
    		\begin{align*}
    			&F_{g,ij}^{C,n+1,k+\half} = \Bar{F}_{g,ij}^{I,n+1} 
    			- (1 - \theta_{g,i}^{n+1,k+\half}) \Bar{F}_{g,ij}^{B,n+1,+}
    			- (1 - \theta_{g,j}^{n+1,k+\half}) \Bar{F}_{g,ij}^{B,n+1,-},
    			\\
    			&F_{g,ij}^{D,n+1,k+\half*} = D_{g,ij}^{n+1,k+\half} \frac{\phi_{j}^{n+1,k+\half} - \phi_i^{n+1,k+\half}}{ |\vr_j - \vr_i| }|S_{ij}|,
    			\\
    			&D_{g,ij}^{n+1,k+\half} =  ( 1 - \theta_{g,ij}^{n+1,k+\half}  ) \frac{1 }{3\sigma_{g,ij}^{n+1,k+\half}}(1-e^{-c \sigma_{g,ij}^{n+1,k+\half}\Delta t}) (b_g + \frac{T}{4} \pdrv{b_g}{T})_{ij}^{n+1,k+\half}.
    		\end{align*}
    	\end{subequations}
    	Notice that we take $\left( b_g + \frac{T}{4} \frac{\partial b_g}{\partial T} \right)_{ij}^{n+1,k+\frac{1}{2}}$ to match with $\frac{\phi_{j}^{n+1,k+\half} - \phi_i^{n+1,k+\half}}{ |\vr_j - \vr_i| }$, while $b_{g,i}^{n+1,k+1}$ matches with $\phi_i^{n+1,k+1}$. This consistent temporal treatment ensures the  relation $  4 \pi \nabla B_g =\left(  b_g + \frac{T}{4} \frac{\partial b_g}{\partial T} \right) \nabla \phi $ and $4 \pi B_g = b_g \phi$ are properly maintained  during  the correction step. In particular, the diffusive flux $F_{g,ij}^{D,n+1,k+\half*}$ is evaluated using information from the prediction step, which helps decouple the space dimension.
    	
    	Utilizing $\phi_i^{n+1,k+1} = ac \, (T_i^{n+1,k+1})^4$, we reformulate \eqref{eq_correction_B} into the following \textit{scalar} nonlinear system:
    	\begin{equation}\label{eq_corretion}
    		\begin{aligned}
    			T_i^{n+1,k+1} +& \frac{ a}{C_{v,i}} \sum_{g=1}^G \chi_{g,i}^{n+1, k+ \half}  b_{g,i}^{n+1, k+1} (T_i^{n+1,k+1})^4  - A_i =0,
    		\end{aligned}
    	\end{equation}
    	where 
    	\begin{equation*}
    		A_i = T_i^n + \frac{\Delta t}{C_{v,i}} \sum_{g=1}^G \chi_{g,i}^{n+1, k+ \half}\left( \frac{1}{c \Delta t} \rho_{g,i}^n- \frac{1}{\Delta V_i}  \sum_{j \in \mathcal{N}_i} 
    		 F_{g,ij}^{n+1,k+\half*} \right).
    	\end{equation*}
    	
    	In this way, updating the temperature $T_i^{n+1,k+1}$ reduces to solving a \textit{spatially decoupled scalar} nonlinear equation within each cell, thereby avoiding the need to solve a global nonlinear system across all cells. 
    	When the original equation reduces to the gray case, \eqref{eq_corretion} simplifies to a polynomial nonlinear system, as  demonstrated in \cite{xiong2022high,tang2021accurate}. Here, we use Newton's iteration to solve these scalar nonlinear equations, with details provided in \ref{appendix_A}.
    	
    	\textbf{The stop criteria.}
    	The two steps are solved with the iterative number $k$ until convergence is reached, where the stop criteria is defined as
    	\begin{equation} \label{eq_stop_criteria}
    		\| T^{n+1.k+1} - T^{n+1.k} \|_1 < \gamma.
    	\end{equation}
    	We note that, in theory, we cannot prove the convergence of the Picard iteration with 
    	\eqref{eq_prediction} and \eqref{eq_corretion}. Nonetheless, we use the $L_1$ norm in our stopping criterion, and numerically all examples converge under the tolerance $\gamma = 10^{-8}$ with the maximum iterative number $50$.

\begin{remark}
    (Numerical boundary treatment)  
     In our numerical tests, we consider only the isotropic  boundary conditions, assuming local equilibrium at the boundary so that $I_{bc} = B(\nu,T)$. Taking 1D as an example, on the left boundary at $x_{\half}$ with given $T_L$, the inflow boundary $I_{bc}$ is set as $I_{bc} = B(\nu,T_{L})$, and $I_{g,bc} = b_g(T_{L})\frac{ac(T_{L})^4}{4\pi}$. 

     For the inflow part of $I_{g,\frac{1}{2}}$, we have
    \begin{equation}
        I_{g,\frac{1}{2}} = I_{g,bc}, \quad \mu >0,
    \end{equation}
    while for the outflow part of $I_{g,\frac{1}{2}}$, we take the approximated formal solution
\begin{equation} \label{eq_bc_formal}
		\begin{aligned}
		  I_{g,\frac{1}{2}}
            = & e^{-c \sigma_{g}(t-t^n)} I_{g}\left( \mfx_{\half}(t^n), \mu,t^n\right) 
            \\ 
		  & +  \theta_{g}     ( 1 - e^{-c \sigma_{g}(t-t^n)}) B_g \left( x_{\half},  t\right) 
            \\
            & + ( 1 - \theta_{g}  ) \Bigg[ B_g  \left( x_{\half},  t\right)-B_g \left( \mfx_{\half}(t^n),  t^n\right) e^{-c \sigma_{g}(t-t^n)} \\ 
			& \quad  -  \frac{1}{c \sigma_{g}} (1 - e^{-c \sigma_{g}(t-t^n)}) \left( \pdrv{B_g }{t}(x_{\half},t) +  c \mu  \pdrv{B_g}{x}  (x_{\half},t)  \right)   \Bigg], \quad \mu <0.
		\end{aligned}
    \end{equation}
    We set $\theta_{g,\half} = \theta_{g,1}$, $\sigma_{g,\half} = \sigma_{g,1}$, and define $\phi_{\half} = \frac{1}{2}(4 \pi I_{bc} + \phi_1^n)$. The temperature $T_{\half}$ is then computed as $T_{\half} = \left(\frac{\phi_{\half}}{ac}\right)^{\frac{1}{4}},$ which is subsequently used to evaluate both $b_{g,\half}$ and $\left(b_g + \frac{T}{4}\frac{\partial b_g}{\partial T}\right)_{\half}$. Substituting  \eqref{eq_bc_formal} into $\left\langle \mu I_{g,\frac{1}{2}} \right\rangle $, where $\left\langle \cdot \right\rangle := 2 \pi \int_{-1}^1 \dd{\mu}$, and applying numerical discretization, we obtain:
 	\begin{equation} \label{eq_bc}
		\begin{aligned}
			\left\langle \mu I_{g,\frac{1}{2}} \right\rangle = &
             I_{g,bc}\left\langle \mu \mathbf{1}_{\mu>0}\right\rangle + ( 1 - \theta_{g,\half} e^{-c \sigma_{g,\frac{1}{2}}(t-t^n)} )  b_{g,\half} \phi_{\half}  \left\langle \mu \mathbf{1}_{\mu<0}\right\rangle
            \\
            &+ \left\langle \mu e^{-c \sigma_{g,\frac{1}{2}}\left(t-t_n\right)} I_{g}\left( \mfx_{\frac{1}{2}}(t^n), \mu,t^n\right)\mathbf{1}_{\mu<0} \right\rangle  
            \\
            &- (1- \theta_{g,\half} )\left\langle \mu  e^{-c \sigma_{g,\frac{1}{2}} \left(t-t^n\right)}  B_g \left( \mfx_{\frac{1}{2}}(t^n),t^n\right)\mathbf{1}_{\mu<0} \right\rangle 
            \\
            &- ( 1 - \theta_{g,\half} ) \frac{\left\langle \mu^2 \mathbf{1}_{\mu<0}\right\rangle}{ 4 \pi \sigma_{g,\frac{1}{2}}}(1-e^{-c \sigma_{g,\frac{1}{2}}\left(t-t_n\right)})  (b_g + \frac{T}{4} \pdrv{b_g}{T})_{\half} \frac{\phi_{\half}  -  \phi_1 }{\Delta x_1 / 2},
		\end{aligned}
	\end{equation}
 where the directional indicator function $\mathbf{1}_{\mu \lessgtr 0}$ takes value $1$ when $\mu \lessgtr 0$ and 0 otherwise. Here, we have $\left\langle \mu\mathbf{1}_{\mu<0}\right\rangle = -\pi$, $\left\langle \mu \mathbf{1}_{\mu>0}\right\rangle = \pi$ and  $\left\langle \mu^2 \mathbf{1}_{\mu<0}\right\rangle = \frac{4 \pi}{6}$.
 
 A similar approach is applied for vacuum boundary condition, and the formulation can be naturally extended to two-dimensional cases along each direction.
\end{remark}

    \subsection{A Monte Carlo method for the microscopic system}
    In this subsection, we present a particle-based MC method for solving the microscopic transport equation. The microscopic evolution also yields the convective flux $F_{g,ij}^{C,n+1}$, which is used to in the finite volume scheme \eqref{eq_fullDiscretization}. Since the emission source $B_{g,i}^{n+1}$ can be determined by the material temperature $T_i^{n+1}$ from the macroscopic system, the microscopic equation \eqref{eq_microscopicSystem_A} reduces to a purely absorbing problem. A MC method for \eqref{eq_microscopicSystem_A} is thus straightforward. Finally, to ensure consistency, the final material temperatures $T_i^{n+1}$ are updated by tallying the results from the MC solution of the microscopic transport equation, rather than using the values from the macroscopic system.
    
    Due to Duhamel’s principle, in each time interval $[t^n,t^{n+1}]$, the microscopic equation \eqref{eq_microscopicSystem_A} can be divided into the following two subsystems:
\begin{equation} \label{eq_Duhammel_A}
\begin{aligned}
 \left\{
\begin{array}{l}
\frac{1}{c} \frac{\partial I_{1,g}}{\partial t} + \vOmega \cdot \nabla I_{1,g} + \sigma_g I_{1,g} = 0, \\[8pt]
I_{1,g}(\vx, \vOmega, t^n) = I_{g}^n(\vx, \vOmega), \\[8pt]
I_{1,g}(\vx, \vOmega, t)\big|_{\vx \in \partial V} = I_{bc,g}(\vx, \vOmega, t), \quad \vOmega \cdot \mathbf{n} < 0,
\end{array}
\right.
\end{aligned}
\end{equation}
and  
\begin{equation} \label{eq_Duhammel_B}
\begin{aligned}
\left\{
\begin{array}{l}
\frac{1}{c} \frac{\partial I_{2,g}}{\partial t} + \vOmega \cdot \nabla I_{2,g} + \sigma_g I_{2,g}= \sigma_g  B_g, \\[8pt]
I_{2,g}(\vx, \vOmega, t^n) = 0, \\[8pt]
I_{2,g}(\vx, \vOmega, t)\big|_{\vx \in \partial V} = 0, \quad \vOmega \cdot \mathbf{n} < 0.
\end{array}
\right.
\end{aligned}
\end{equation}
Here $I_{g}^n$ and $I_{bc,g}$ represent the group-wise radiation intensity 
at time $t^n$ and on the boundary, respectively. 
Systems \eqref{eq_Duhammel_A} and \eqref{eq_Duhammel_B} indicate that 
the total intensity originates from two distinct sources. 
The first source consists of photons from the previous time step 
(or initial condition) and the boundary, which are already known and can be tracked immediately. 
The second source corresponds to the unknown photons emitted by the material, 
which can be calculated with $T^{n+1}$ from  the macroscopic equation.

    In a particle based method, the group-wise radiation intensity $I_g$ is represented as a collection of particles, and can be expressed as
    \begin{equation*}
I_g(\vx, \vOmega, t)=\sum_{p=1}^{N(t)} c \omega_p^g(t) \delta(\vx-\vx_p(t)) \delta(\vOmega-\vOmega_p(t)),
\end{equation*}
where $\vx_p(t)$, $\vOmega_p(t)$ and $\omega_p^g(t)$ are the location, angular direction and group-wise energy weight of particle $p$ at time $t$, respectively. $N(t)$ is the total number of MC particles used at time $t$. 

The essence of the particle based MC method lies in the fact that each particle is represented as a quadruple $(\vr_p,\vOmega_p,\omega_p^g,t_p)$, which is used to mimic the transport and absorption behavior. We can then recover the necessary physical quantities (related to the solution of the original PDE) from these ensembles of particles.


 \subsubsection{Particle sampling}
To solve systems \eqref{eq_Duhammel_A} and \eqref{eq_Duhammel_B}, one needs to sample 
MC particles from the previous time step  (or the initial condition), 
the boundary condition and the emission source. 
Each particle is assigned a position, angular direction, time, 
and  energy weight. 
Details on how to sample position, angular direction, and time 
can be found in \cite{fleck1971implicit,wollaber2016four}. 
Here, we focus on the way to compute the corresponding energy 
and assign energy weights to ensure energy conservation.


 For particles from the initial condition or the previous time step,  the function $I_g^n$ is employed for sampling. The radiation energy for this portion of particles can be obtained by integrating the radiation intensity $I^n_g$ over the cell and angle
\begin{equation}\label{eq_sample_I}
    E_{g,i}^{I,n} = \frac{1}{c} \int_{V_i} \int_{4\pi} I^n_g(\vx,\vOmega) \dd{\vOmega} \dd{\vx}.
\end{equation}
Similarly, The total energy due to the boundary condition is obtained by integrating the function $I_{bc,g}$ over the time step, boundary surface, and angle corresponding to an inflow part
\begin{equation}\label{eq_sample_bc}
    E_{g}^{I,bc} = \int_{t^{n}}^{t^{n+1}}   \int_{\vOmega\cdot\mathbf{n}<0 } \int_{\partial V} ( - \vOmega\cdot\mathbf{n}) I_{bc,g}(\vx,\vOmega,t) \dd{S} \dd{\vOmega}  \dd{t}.
\end{equation}
However, the boundary source is frequently
specified as a Planck function at a fixed temperature, i.e., $I_{bc,g}(\vx,\vOmega,t) = B_g\left(T_{bc}\left( \vx\right)\right)$, in which case the total energy may be written more simply as
\begin{equation*}
    E_{g}^{I,bc} = \pi  \Delta t   \int_{\partial V} B_g(T_{bc} \left( \vx\right) )  \dd{S}.
\end{equation*}
As noted from the last subsection, the convective flux $F^{C,n+1}_{ij}$, given by
\begin{equation*}         
        \begin{aligned}  
            &F_{g,ij}^{C,n+1} =  
            \Bar{F}_{g,ij}^{I,n+1} 
            - (1 - \theta_{g,i}^{n+1}) \Bar{F}_{g,ij}^{B,n+1,+}
            - (1 - \theta_{g,j}^{n+1}) \Bar{F}_{g,ij}^{B,n+1,-},
        \end{aligned}
    \end{equation*}
is obtained by the MC method.
Those photons from the source \eqref{eq_sample_I} and \eqref{eq_sample_bc} contributes to the first term $\Bar{F}_{g,ij}^{I,n+1} $. While for $\Bar{F}_{g,ij}^{B,n+1,+}$ and $\Bar{F}_{g,ij}^{B,n+1,-}$, these photons can be calculated similarly, given by the ghost sources
\begin{equation}\label{eq_sample_B}
    E_{g,i}^{B,n} = \frac{1}{c} \int_{V_i} \int_{4\pi} B^n_g(\vx) \dd{\vOmega} \dd{\vx},
\end{equation}
and
\begin{equation}\label{eq_sample_B_bc}
    E_{g}^{B,bc} = \int_{t^{n}}^{t^{n+1}}   \int_{\vOmega\cdot\mathbf{n}<0 } \int_{\partial V} ( - \vOmega\cdot\mathbf{n}) B_{bc,g}^n(\vx) \dd{S} \dd{\vOmega}  \dd{t}.
\end{equation}
The definition of $B_{bc,g}^n(\vx)$ is consistent with the boundary treatment for the macroscopic equation. In 1D example, $B_{bc,g}^n(\vx)$ corresponds specifically to the term $b_{g,\half}\phi_{\half}$ defined in the first line of equation \eqref{eq_bc}. We emphasize that photons from ghost sources are used exclusively for tallying the convective flux $F^{C,n+1}_{ij}$—they neither enter the census\footnote{In radiative transfer terminology, this typically refers to particles advancing to the next time step.} nor contribute to absorption energy calculations.

Next, we consider the MC particles emitted from the material. For these particles, the radiation energy is computed from the macroscopic variables, which is

\begin{equation}\label{eq_sample_em}
E_{g,i}^{R,n+1}=\int_{t_n}^{t_{n+1}} \int_{V_i} \int_{4 \pi} \sigma_g^{n+1} B_g^{n+1}\dd{\vOmega} \dd{\vx} \dd{t}=   4 \pi \sigma_{g,i}^{n+1} B_{g,i}^{n+1} \Delta V_i \Delta t,
\end{equation}
where $B_{g,i}^{n+1}$ is the cell average evaluated using $T^{n+1}_i$ from the macroscopic equations. As shown in  \cite{fleck1984random,shi2020continuous,densmore2011asymptotic}, in order to capture the equilibrium diffusion limit, a linear representation of the emission source is necessary. In the following, we will present a continuous \textit{source tilting} method for the emission source, extending our previous work \cite{shi2020continuous} for the frequency-dependent radiative transfer equations.

For illustration, we consider the 1D case, while the extension to 2D is available in \cite{shi2020continuous} and is straightforward along each direction. In 1D, for each cell
 $V_i$ centered at $x_i$ with mesh size $\Delta x_i$, we define the  linear reconstruction of $ (B_{tilt})^{n+1}_{g,i}$ as 
\begin{equation}\label{eq_sourceTilting}
(B_{tilt})^{n+1}_{g,i}(x) = B^{n+1}_{g,i} + 
\begin{cases} 
s^{B}_{g,i } \left( x - x_i \right), & \text{if } \mu < 0, \\
s^{F}_{g,i } \left( x - x_i \right), & \text{if } \mu > 0,
\end{cases}
\end{equation}
where $B_{g,i}^{n+1}$ is the cell average. The backward and forward one-sided slopes are given by
\begin{equation*}
s^{B}_{g,i } = \frac{B_{g,i}^{n+1} - B^{n+1}_{g,i - 1}}{\frac{1}{2}( \Delta x_i   + \Delta x_{i-1} ) }, \quad s^{F}_{g,i } = \frac{B^{n+1}_{g,i + 1} - B^{n+1}_{g,i}}{\frac{1}{2}( \Delta x_{i+1}   + \Delta x_{i} )}.
\end{equation*}
With this definition, the particle positions drawn from the emission source in cell $V_i$ follow the probability distribution function
\begin{equation}\label{eq_sourceTilting_PDF}
    p_{g}(x)\Big|_{V_i} = \frac{1}{\Delta x_i} \frac{(B_{tilt})_{g,i}^{n+1}(x)}{ B^{n+1}_{g,i}}.
\end{equation}
To ensure positivity of the probability distribution function, the slopes must satisfy \cite{densmore2011asymptotic}
\begin{equation*}
    \lvert s^{B}_{g,i } \rvert \leq \frac{2B^{n+1}_{g,i}}{\Delta x_i}, \quad
     \lvert s^{F}_{g,i } \rvert \leq \frac{2B^{n+1}_{g,i}}{\Delta x_i}.
\end{equation*}

We remark that equation \eqref{eq_sourceTilting} is used to bias the distribution of the locations for the emission photons via \eqref{eq_sourceTilting_PDF} (replacing uniform sampling in $V_i$), while the total emission source strength \eqref{eq_sample_em} remains determined directly by the cell-averaged values $B_{g,i}^{n+1}$.

\subsubsection{Particle tracking}
 After all MC particles have been sampled, the subsequent task involves tracking each particle's trajectory.  As previously noted, the system transitions to a purely absorbing scenario once the emission source is established, simplifying the tracking process. Three fundamental events govern particle trajectories:  (i) absorption by material, (ii) traversal across a cell interface, or (iii) survival until reaching the end of the time step at  $t^{n+1} $. Each event corresponds to a distinct characteristic distance:  the absorption distance $d_A$, the boundary distance $d_B$ to the cell interface, and the temporal survival distance  $d_T$. The distance to the cell interface $d_B$ satisfies
 \begin{equation*}
     \vx_B - \vx_p = d_B \vOmega_p,
 \end{equation*}
 where $\vx_p$ is the location of particle $p$, $\vx_B$ is the cell interface location in direction $\vOmega_p$. The temporal survival distance $d_T$  is
 \begin{equation*}
     d_T = c (t^{n+1} - t),
 \end{equation*}
 where $t$ is the current time of each particle and $c$ is the speed of light. For the absorption event, we employ the \textit{continuous energy deposition} variance reduction technique  \cite{fleck1971implicit}. In this approach, the absorption distance $d_A$ is implicitly determined through exponential decay  of the energy weight rather than explicit calculation.
In summary, if we let 
\begin{equation*}
    d = \min{(d_B,d_T)},
\end{equation*}
the particle $p$ is advanced according to
\begin{equation*}
    \begin{aligned}
        &\vx_p^{'} = \vx_p + \vOmega_p d,
        \\
        &\vOmega_p^{'} = \vOmega_p,
        \\
        &t^{'} = t + d/c,
        \\ \label{CED}
        &\omega_p^{g'} = \omega_p^{g} e^{-\sigma_{g} d},
    \end{aligned}
\end{equation*}
where $\vx_p^{'}$, $\vOmega_p^{'}$, $t^{'}$ and $\omega_p^{g'} $  denote the new location, direction, time and group-wise energy weight of each particle, respectively. Thanks to $d_B$, we partition the tracking step for each spatial cell, where the group-wise absorption opacity $\sigma_g$ is assumed to be constant.

The tracking process for each particle continues until one of three termination conditions is met: (i)  the particle's current energy weight falls below  $ 0.01\% $  of its initial birth weight, (ii) the particle leaks out of the physical domain $V$, or (iii) the particle enters census.

\subsubsection{Tally}
During the evolution of particle trajectory, the following three quantities need to be tallied. 
\begin{itemize}
    \item The first quantity is the radiation energy in each cell at at the new time step. This value can be evaluated by
\begin{equation} \label{eq_def_EI}
    E_{g,i}^{I,n+1}  = \sum_{p=1}^M \omega_p^g(t^{n+1}),
\end{equation}
where $p= 1,2,\ldots,M$ denotes the number of particles that go to census in the cell $V_i$. We note that only the particles from the radiation sources \eqref{eq_sample_I} \eqref{eq_sample_bc} and \eqref{eq_sample_em} are used for tallying this quantity. The group-wise angular integrated intensity at the new time step can then be calculated by
\begin{equation}\label{eq_tally_Tr2}
	\rho_{g,i}^{n+1} = \frac{c E_{g,i}^{I,n+1} }{\Delta V_i} .
\end{equation}
Additionally, we use this quantity to compute the radiation temperature:
\begin{equation}\label{eq_tally_Tr}
	T_{r,i}^{n+1} = \left( \frac{\sum_{g=1}^G E_{g,i}^{I,n+1} }{a \Delta V_i}\right )^{\frac{1}{4}}.
\end{equation}
\item The second quantity is the radiation energy deposited due to absorption in each cell during $[t^n,t^{n+1}]$, which is given by
\begin{equation}\label{eq_def_EA}
    E_{g,i}^{A,n+1} := \frac{1}{c} \int_{t^n}^{t^{n+1}}  \int_{V_i} \int_{4 \pi} c \sigma_g I_g \dd{\vOmega} \dd{\vx} \dd{t} =  \sum_{p=1}^N \omega_p^{g} (1 - e^{-\sigma_{g,i} d} ),
\end{equation}
where $p= 1,2,\ldots,N$ denotes the number of particles traversing cell $V_i$ during  $[t^n,t^{n+1}]$, and $d$ represents the traveling distance within $V_i$ over $[t^n,t^{n+1}]$. We emphasize that only particles originating from the sources in \eqref{eq_sample_I}, \eqref{eq_sample_bc}, and \eqref{eq_sample_em} are considered for tallying this quantity. This value is used in updating the material temperature $T_i^{n+1}$ by integrating equation \eqref{eq_microscopicSystem_B} over $V_i$ and $[t^n,t^{n+1}]$:
\begin{equation} \label{eq_tally_Tm}
    T_i^{n+1} = T_i^n + \frac{1}{C_{v,i}} \frac{1}{ \Delta V_i}    \sum_{g=1}^G \left  (E_{g,i}^{A,n+1} -  E_{g,i}^{R,n+1} \right).
\end{equation}
To ensure consistency, $T_i^{n+1}$ is updated for the next time step using \eqref{eq_tally_Tm}, rather than adopting the quantities from the macroscopic system.

\item  The final quantities to be computed are the convective fluxes across the surface $S_{ij} $ during  $[t^n,t^{n+1}]$. They are computed by
\begin{equation} \label{eq_def_fluxA}
    \Bar{F}_{g,ij}^{I,n+1} = \frac{1}{\Delta t} \sum_{p=1}^K \text{sign}(\vOmega_p \cdot \mathbf{n}_{ij})\omega_p^g,
 \end{equation}
 where $p= 1,2,\ldots,K$  denotes the number of particles originated from the sources \eqref{eq_sample_I} and \eqref{eq_sample_bc} that traverse the  interface $S_{ij}$ during $[t^n,t^{n+1}]$. And
 \begin{equation}\label{eq_def_fluxB}
     \begin{aligned}
        &\Bar{F}_{g,ij}^{B,n+1,+} = \frac{1}{\Delta t} \sum_{p=1}^{L} \text{sign}(\vOmega_p \cdot \mathbf{n}_{i,j})\omega_p^g \mathbf{1}_{\vOmega_p \cdot \mathbf{n}_{ij}>0},
        \\
        &\Bar{F}_{g,ij}^{B,n+1,-} = \frac{1}{\Delta t} \sum_{p=1}^{L} \text{sign}(\vOmega_p \cdot \mathbf{n}_{i,j})\omega_p^g \mathbf{1}_{\vOmega_p \cdot \mathbf{n}_{ij}<0},
     \end{aligned}
 \end{equation}
  where $p= 1,2,\ldots,L$  denotes the number of particles originated from the ghost sources \eqref{eq_sample_B} and \eqref{eq_sample_B_bc} that traverse the interface $S_{ij}$ during $[t^n,t^{n+1}]$ . 
\end{itemize}

\subsection{An overall algorithm}
Finally we present our updating procedure from $t^n$ to $t^{n+1}$ in Algorithm \eqref{eq_algorithm}.

\begin{algorithm}[H]
	\caption{The updating procedure from $t^n$ to $t^{n+1}$.}
	\label{eq_algorithm}
	\begin{algorithmic}[1]
		\STATE Evaluate multi-group opacity $\sigma_g(T^n_i)$;
		\STATE Sample MC particles from sources $ E_{g,i}^{I,n}$, $ E_{g}^{I,bc}$, $ E_{g,i}^{B,n}$, $ E_{g}^{B,bc}$;
		\STATE Track particle trajectories,
		\begin{itemize}
			\item For particles from $ E_{g,i}^{I,n}$, $ E_{g}^{I,bc}$, tally the convective flux $\Bar{F}_{g,ij}^{I,n+1}$, the radiation energy $E_{g,i}^{I,n+1}$ and the absorbed energy $E_{g,i}^{A,n+1} $; this resolves \eqref{eq_Duhammel_A};
			\item For particles from $ E_{g,i}^{B,n}$, $ E_{g}^{B,bc}$, tally only the convective flux $\Bar{F}_{g,ij}^{B,n+1,+}$ and $\Bar{F}_{g,ij}^{B,n+1,-}$;
		\end{itemize}
		\STATE Solve the macroscopic system \eqref{eq_fullDiscretization} with 
    	\eqref{eq_prediction} and \eqref{eq_corretion} until the stop criteria \eqref{eq_stop_criteria} is reached. Once $T_i^{n+1}$ is obtained, update $B_{g,i}^{n+1}$ accordingly;
		\STATE Evaluate multi-group opacity $\sigma_g(T^{n+1}_i)$;
		\STATE Reconstruct linear source term  $(B_{tilt})_{g,i}^{n+1}$;
		\STATE Sample MC particles from the emission source $E_{g,i}^{R,n+1}$;
		\STATE Track particle trajectories, tally the radiation energy $E_{g,i}^{I,n+1}$ and the absorbed energy $E_{g,i}^{A,n+1} $; this resolves \eqref{eq_Duhammel_B};
		\STATE Update group-wise angular integrated intensity $\rho_{g,i}^{n+1}$ and material temperatures $T_i^{n+1}$ using \eqref{eq_tally_Tr2} and \eqref{eq_tally_Tm}, respectively.
	\end{algorithmic}
\end{algorithm}

\section{Formal asymptotic analysis}
\label{sec:analysis}

In this section, we will formally prove the proposed numerical method preserves the asymptotic property in the equilibrium diffusion limit. 

Let  $\varepsilon > 0$ denote the dimensionless Knudsen number, defined as the ratio of the mean free path to the characteristic length scale of the system. If the system
is optically thick, and the speed of light is fast compared to the time evolution of $I$, the  radiative transfer equations \eqref{eq_RTE} can be rewritten in the following scaled form \cite{larsen2013properties}:
\begin{equation} \label{eq_RTE_scaled}
    \begin{aligned}
		&\frac{\varepsilon^2}{c} \frac{\partial I}{\partial t}+\varepsilon \vOmega \cdot \nabla I=\sigma (B -I), 
             \\ 
		&\varepsilon^2 C_v \pdrv{T}{t}=\int_0^{\infty} \int_{4 \pi} \sigma(I-B) \dd{\vOmega}\dd{\nu},
    \end{aligned}
\end{equation}
where the opacity, heat capacity, and speed of light are scaled as
\begin{equation}\label{eq_diffusiveScale}
\sigma  \rightarrow \frac{\sigma}{\varepsilon}, \quad c \rightarrow \frac{c}{\varepsilon}, \quad C_v \rightarrow \varepsilon C_v,
\end{equation}
respectively. We note that the scale of $c$ does not apply to the emission source term because this term is a Planckian at the local material temperature, which does not change in the equilibrium diffusion limit. In \cite{larsen1987asymptotic,larsen1983asymptotic}, Larsen et al. have shown that away from boundaries and initial times, as $\varepsilon \rightarrow 0$, the leading order radiation intensity $I^{(0)}$ approaches to a Planckian at the local temperature, 
		\begin{equation*}
			I^{(0)} = B(\nu,T^{(0)}),
		\end{equation*}
		and the leading order material temperature $T^{(0)}$ satisfies the following radiation diffusion equation
		\begin{equation}\label{eq_def_diffLimit}
			a \pdrv{}{t} (T^{(0)})^4 + C_v \pdrv{}{t}T^{(0)} = \nabla \cdot \left( \frac{ac}{3 \sigma_R} \nabla (T^{(0)})^4 \right),
		\end{equation}
		with the Rosseland mean opacity $\sigma_R$ given by 
		\begin{equation}\label{eq_def_Rosseland}
			\frac{1}{\sigma_R} 
			= \frac{\int_0^{\infty} \frac{1}{\sigma}  \pdrv{B(\nu,T^{(0)})}{T} \dd{\nu}}{ \int_0^{\infty}    \pdrv{B(\nu,T^{(0)})}{T}\dd{\nu} }.
		\end{equation}

Next, We will analyze the asymptotic behavior of the proposed method by expanding the group-integrated radiation intensity and the material temperature in powers of $\varepsilon$:
   \begin{equation*}
       I_g = \sum_{k=0}^{\infty} \varepsilon^{k} I_g^{(k)},
   \end{equation*}
   and
   \begin{equation*}
       T = \sum_{k=0}^{\infty} \varepsilon^{k} T^{(k)},
   \end{equation*}
   and compare terms that are the same order in $\varepsilon$.
   The temperature-dependent terms can also be expanded into a power series in $\varepsilon$. For example, the group integrated Planck function $B_g(T)$ can be written as
   \begin{equation*}
       B_g = B_g^{(0)} + \varepsilon B_g^{(1)} + \cdots,
   \end{equation*}
   where 
   \begin{equation*}
       \begin{aligned}
           &B_g^{(0)} = B_g |_{\varepsilon = 0} = B_g|_{T = T^{(0)}},
           \\
           &B_g^{(1)} = \pdrv{B_g}{\varepsilon}\Big|_{\varepsilon = 0} = \pdrv{B_g}{T}\pdrv{T}{\varepsilon}\Big|_{\varepsilon = 0} = \pdrv{B_g}{T}\Big|_{T = T^{(0)}} T^{(1)}.
       \end{aligned}
   \end{equation*}

   We first show the multi-group discretization using piecewise constant approximation is asymptotic preserving, this proposition comes from \cite{ZHANG2023112368}. For the sake of readability, we include the proof for Proposition \ref{eq_prpo1} in the appendix.
   
\begin{proposition} \label{eq_prpo1}
    When $\varepsilon$ tends to $0$, the limit of the  multi-group discretization of the scaled radiative transfer equation \eqref{eq_RTE_scaled} utilizing piecewise constant approximation \eqref{eq_Approximation_sigma} approaches to the radiation diffusion equation \eqref{eq_def_diffLimit}.
\end{proposition}

We now state a useful lemma.
\begin{lemma}
    When $\varepsilon$ tends to $0$, we have:
    \begin{itemize}
        \item $\theta^{n+1}_{g,i} = e^{-c \sigma_i^{n+1} \Delta t / \varepsilon^2} \longrightarrow 0$, 
        \item $\theta^{n+1}_{g,ij} =  \half(\theta_{g,i}^{n+1} + \theta_{g,j}^{n+1}   ) \longrightarrow 0$, 
        \item $ D_{g,ij}^{n+1} = ( 1 - \theta^{n+1}_{g,ij}  ) \frac{\varepsilon }{3\sigma^{n+1}_{g,ij}}(1-e^{-c \sigma_{g,ij}^{n+1}\Delta t/ \varepsilon^2}) (b_g + \frac{T}{4} \pdrv{b_g}{T})_{ij}^{n+1} \longrightarrow \frac{\varepsilon }{3\sigma_{g,ij}^{n+1}} (b_g + \frac{T}{4} \pdrv{b_g}{T})_{ij}^{n+1} $ , 
        \item $ \chi_{g,i}^{n+1}  = \frac{\sigma_{g,i}^{n+1} }{\frac{\varepsilon^2}{c \Delta t} + \sigma_{g,i}^{n+1} } \longrightarrow 1 $.
    \end{itemize}
\end{lemma}
As a consequence, the convective flux $F^{C,n+1}_{g,ij} $ defined in \eqref{eq_fulldis_fluxA} has the following limit:
\begin{equation*}
	F^{C,n+1}_{g,ij} = \Bar{F}_{g,ij}^{I,n+1} 
	- (1 - \theta_{g,i}^{n+1,k}) \Bar{F}_{g,ij}^{B,n+1,+}
	- (1 - \theta_{g,j}^{n+1,k}) \Bar{F}_{g,ij}^{B,n+1,-} \underset{\varepsilon \rightarrow 0}{\longrightarrow}  0,
\end{equation*}
while the diffusive flux $F_{g,ij}^{D,n+1}$ defined in \eqref{eq_def_diffusiveFLux} scaled by the factor $\frac{1}{\varepsilon}$  has the following limit: 
\begin{equation*}
	\frac{1}{\varepsilon} F_{g,ij}^{D,n+1} = \frac{1}{\varepsilon} D_{g,ij}^{n+1} \frac{\phi_{j}^{n+1} - \phi_i^{n+1}}{ |\vr_j - \vr_i| }|S_{ij}|  \underset{\varepsilon \rightarrow 0}{\longrightarrow} \frac{1 }{3\sigma_{g,ij}^{n+1}} (b_g + \frac{T}{4} \pdrv{b_g}{T})_{ij}^{n+1} \frac{\phi_{j}^{n+1} - \phi_i^{n+1}}{ |\vr_j - \vr_i| }|S_{ij}|.
\end{equation*}

Next, we show that the full discretization, using Picard iteration with a predictor-corrector procedure, possesses the following asymptotic behavior.
\begin{proposition}
    When $\varepsilon$ tends to $0$, in  the prediction step, the limit of the  full discretization of the scaled radiative transfer equation \eqref{eq_RTE_scaled}  approaches to the implicit scheme for the semilinear diffusion equation for $(T^{(0)})^4:$
    \begin{equation} \label{eq_quasi}
    		( \frac{C_v}{4(T^{(0)})^3} + a) \pdrv{}{t} (T^{(0)})^4  = \nabla \cdot \left( \frac{ac}{3 \sigma_R} \nabla (T^{(0)})^4 \right), 
    \end{equation}
    while in  the correction step, the limit approaches to the implicit scheme for the nonlinear diffusion equation for $T^{(0)}:$
        \begin{equation} \label{eq_non}
       		a \pdrv{}{t} (T^{(0)})^4 + C_v \pdrv{}{t}T^{(0)} = \nabla \cdot \left( \frac{ac}{3 \sigma_R} \nabla (T^{(0)})^4\right).
    \end{equation}
\end{proposition}

\begin{proof}
	In the prediction step \eqref{eq_prediction}, with the scale \eqref{eq_diffusiveScale}, we have the scheme:
    \begin{equation}
\begin{aligned}
    \Big(\frac{1}{\beta_i^{n+1,k} \Delta t} 
    &+ \frac{1}{c \Delta t} \sum_{g=1}^G \chi_{g,i}^{n+1, k} b_{g,i}^{n+1, k}    \Big)\phi_i^{n+1,k+\frac{1}{2}}
    \\
    &-   \frac{1}{ \varepsilon \Delta V_i } \sum_{g=1}^G \chi_{g,i}^{n+1, k} \left(  \sum_{j \in \mathcal{N}_i}   D_{g,ij}^{n+1,k}\frac{\phi_{j}^{n+1,k+\half} - \phi_i^{n+1,k+\half}}{ |\vr_j - \vr_i| }|S_{ij}|    \right)= \text{RHS}_i,
\end{aligned}
\end{equation}
where 
\begin{equation*}
    \text{RHS}_i = \frac{1}{\beta_i^{n+1,k} \Delta t}\phi_i^n +  
    \sum_{g=1}^G \chi_{g,i}^{n+1, k}  \left(\frac{1}{c \Delta t}\rho_{g,i}^n- \frac{1}{ \varepsilon \Delta  V_i } \sum_{j \in \mathcal{N}_i} F_{g,ij}^{C,n+1,k} \right).
\end{equation*}

    In the limit  $\varepsilon \rightarrow 0$ with Chapman-Enskog expansion, using the lemma above, the system reduces to:
        \begin{equation}
\begin{aligned}
    &\frac{1}{\beta_i^{n+1,k,(0)} \Delta t} (\phi_i^{n+1,k+\frac{1}{2},(0)} - \phi_i^{n,(0)})
    + \frac{1}{c \Delta t} \sum_{g=1}^G \left(  b_{g,i}^{n+1, k,(0)}    \phi_i^{n+1,k+\frac{1}{2},(0)} - \rho_{g,i}^{n,(0)} \right)
    \\
    &=    \frac{1}{ \Delta  V_i } \sum_{j \in \mathcal{N}_i}  \left(     \sum_{g=1}^G \frac{1 }{3\sigma_{g,ij}^{n+1,k,(0)}} (b_g + \frac{T}{4} \pdrv{b_g}{T})_{ij}^{n+1,k,(0)} \frac{\phi_{j}^{n+1,k+\half,(0)} - \phi_i^{n+1,k+\half,(0)}}{ |\vr_j - \vr_i| }|S_{ij}|    \right).
\end{aligned}
\end{equation}
Using the relations: $\beta = \frac{C_v}{4 a c T^3}$, $\phi = ac T^4$, $\sum_{g=1}^G b_g = 1$, $\sum_{g=1}^G \pdrv{B_g}{T} = \frac{4acT^3 }{4\pi}$, and $b_g + \frac{T}{4} \pdrv{b_g}{T} = \frac{4 \pi}{4acT^3}\pdrv{B_g}{T}$, we have 
\begin{equation} \label{eq_quasiLinear}
	\begin{aligned}
		\left(\frac{C_v}{4 \left(T_i^{n+1,k,(0)}\right)^3} +  a\right) &\frac{\left(T_i^{n+1,k+\frac{1}{2},(0)}\right)^4 - \left(T_i^{n,(0)}\right)^4  }{\Delta t}
		\\
		&=    \frac{1}{ \Delta  V_i } \sum_{j \in \mathcal{N}_i}  \left( \frac{ac }{3\sigma_{R,ij}^{n+1,k,(0)}}  \frac{\left(T_j^{n+1,k+\frac{1}{2},(0)}\right)^4 - \left(T_i^{n+1,k+\frac{1}{2},(0)}\right)^4 }{ |\vr_j - \vr_i| }|S_{ij}|    \right),
	\end{aligned}
\end{equation}
where the Rosseland mean opacity is defined as:
\begin{equation*}
	\frac{1}{\sigma_{R,ij}^{n+1,k,(0)}}= \left( \frac{\sum_{g=1}^{G} \frac{1}{\sigma_g} \pdrv{B_g}{T} }{\sum_{g=1}^{G} \pdrv{B_g}{T} } \right)^{n+1,k,(0)}_{ij}.
\end{equation*}
This provides a consistent approximation to the diffusion equation \eqref{eq_quasi}.   Note that \eqref{eq_quasiLinear} remains valid only when the equilibrium condition: $\sum_{g=1}^{G} \rho_{g,i}^{n,(0)}  = \sum_{g=1}^{G} b_{g,i}^{n,(0)}\phi_{i}^{n,(0)}$ is satisfied. As $\rho_{g,i}^{n}$ is computed via the MC method, the equilibrium property of the MC method is necessary, which will be demonstrated later.

In the correction step \eqref{eq_corretion}, with the scale \eqref{eq_diffusiveScale}, we have the scheme:
\begin{equation}
	\begin{aligned}
		T_i^{n+1,k+1} +& \frac{ a}{C_v} \sum_{g=1}^G \chi_{g,i}^{n+1, k+ \half}  b_{g,i}^{n+1, k+1} (T_i^{n+1,k+1})^4 -A_i =0,
	\end{aligned}
\end{equation}
where 
\begin{equation*}
	\begin{aligned}
		A_i  = T_i^n &+ \frac{\Delta t}{C_v} \sum_{g=1}^G \chi_{g,i}^{n+1, k+ \half}\left( \frac{1}{c \Delta t} \rho_{g,i}^n- \frac{1}{\varepsilon \Delta V_i}  \sum_{j \in \mathcal{N}_i} 
		 F_{g,ij}^{C,n+1,k+\half} \right)
		\\
		&+\frac{\Delta t}{C_v} \sum_{g=1}^G \chi_{g,i}^{n+1, k+ \half}\left( \frac{1}{\varepsilon \Delta V_i}  \sum_{j \in \mathcal{N}_i} 
		D_{g,ij}^{n+1,k+\half}\frac{\phi_{j}^{n+1,k+\half} - \phi_i^{n+1,k+\half}}{ |\vr_j - \vr_i| }|S_{ij}|\right).
	\end{aligned}
\end{equation*}
In the asymptotic limit  $\varepsilon \rightarrow 0$, application of the Chapman-Enskog expansion yields, through analogous arguments to those previously established:
    \begin{equation} \label{eq_nonLinear}
    	\begin{aligned}
    		\frac{C_v}{ \Delta t} (T_i^{n+1,k+1,(0)} - T_i^{n,(0)})
    		&+  a \frac{\left(T_i^{n+1,k+1,(0)}\right)^4 - \left(T_i^{n,(0)}\right)^4  }{\Delta t}
    		\\
    		&=    \frac{1}{ \Delta  V_i } \sum_{j \in \mathcal{N}_i}  \left( \frac{ac }{3\sigma_{R,ij}^{n+1,k+\half,(0)}}  \frac{\left(T_j^{n+1,k+\frac{1}{2},(0)}\right)^4 - \left(T_i^{n+1,k+\frac{1}{2},(0)}\right)^4 }{ |\vr_j - \vr_i| }|S_{ij}|    \right),
    	\end{aligned}
    \end{equation}
    where the Rosseland mean opacity is defined as:
    \begin{equation*}
    	\frac{1}{\sigma_{R,ij}^{n+1,k+\half,(0)}}= \left( \frac{\sum_{g=1}^{G} \frac{1}{\sigma_g} \pdrv{B_g}{T} }{\sum_{g=1}^{G} \pdrv{B_g}{T} } \right)^{n+1,k+\half,(0)}_{ij}.
    \end{equation*}
   This provides a consistent approximation to the diffusion equation \eqref{eq_non}.
\end{proof}

For simplicity, we restrict our analysis to the one-dimensional case with a uniform cell $V_i$ with length $\Delta x$ in subsequent discussions.
\begin{proposition}
    When $\varepsilon$ tends to $0$, the solutions of the microscopic Monte Carlo method can capture the equilibrium diffusion limit \eqref{eq_def_diffLimit}.
\end{proposition}

\begin{proof}
Consider the systems
    \begin{subequations}
		\begin{align}\label{eq_asymptoticAnalysis3A}
            &\begin{aligned}
                \frac{\varepsilon^2}{c} \frac{\partial I_g}{\partial t}+\varepsilon\mu \pdrv{I_g}{x}  = \sigma_{g} B_g-\sigma_{g}I_g,  \quad g=1,\ldots,G, 
            \end{aligned}
            \\
		  &\begin{aligned}\label{eq_asymptoticAnalysis3B}
                \varepsilon^2 C_v \frac{\partial T}{\partial t}=\sum_{g=1}^{G} \sigma_{g} \left( \rho_g- 4 \pi B_g  \right) . 
            \end{aligned}
		\end{align}
   \end{subequations} 
  We now perform a Chapman-Enskog expansion and compare terms that are the same order in $\varepsilon$.

   The $O(1)$ equation for \eqref{eq_asymptoticAnalysis3A} is
    \begin{equation}\label{eq_asymptoticAnalysis3_O1A}
        I_g^{(0)} = B_g^{(0)},  \quad g=1,\ldots,G, 
    \end{equation}
    Integrating \eqref{eq_asymptoticAnalysis3_O1A} over the angular variables yields 
      \begin{equation}\label{eq_asymptoticAnalysis3_O1B}
        \rho_g^{(0)} =  4 \pi B_g^{(0)},  \quad g=1,\ldots,G. 
    \end{equation}
    The $O(\varepsilon)$ equation for \eqref{eq_asymptoticAnalysis3A} is
    \begin{equation}\label{eq_asymptoticAnalysis3_O2A}
        \mu \pdrv{I_g^{(0)}}{x}  + \sigma_{g}I_g^{(1)} = \sigma_{g} B_g^{(1)} ,  \quad g=1,\ldots,G, 
    \end{equation}
    substituting \eqref{eq_asymptoticAnalysis3_O1A} into \eqref{eq_asymptoticAnalysis3_O2A} , we can get
    \begin{equation} \label{eq_asymptoticAnalysis3_O2B}
        \begin{aligned}
                I_g^{(1)} = - \frac{1}{ \sigma_g} \mu \pdrv{B_g^{(0)}}{x}+  B_g^{(1)} ,  \quad g=1,\ldots,G. 
        \end{aligned}
    \end{equation}
    The $O(\varepsilon^2)$ equation for \eqref{eq_asymptoticAnalysis3A} is
    \begin{equation}\label{eq_asymptoticAnalysis3_O3A}
        \frac{1}{c} \frac{\partial I_g^{(0)}}{\partial t}+\mu \pdrv{I_g^{(1)}}{x}  = \sigma_{g} B_g^{(2)}-\sigma_{g}I_g^{(2)},  \quad g=1,\ldots,G, 
    \end{equation}
       The $O(\varepsilon^2)$ equation for \eqref{eq_asymptoticAnalysis3B} is
        \begin{equation}\label{eq_asymptoticAnalysis3_O3B}
        C_v \frac{\partial }{\partial t} T^{(0)} =\sum_{g=1}^{G} \sigma_{g} \left( \rho_g^{(2)}- 4 \pi B_g^{(2)}  \right), 
    \end{equation}

    Up to now, the asymptotic analysis was performed without considering  the discretized formulation of the emission source $(B_{tilt})_{g,i}^{(0)}$.
      From the tilting source definition in \eqref{eq_sourceTilting}, we obtain the second-order accurate approximation:
        \begin{equation}
      	B_{g,i}^{(0)} = (B_{tilt})_{g,i}^{(0)} + O(\Delta x^2).
      \end{equation}
      Consequently, the $O(\varepsilon)$ equation in \eqref{eq_asymptoticAnalysis3A} takes the modified form:
    \begin{equation*}
        \begin{aligned}
                I_{g,i}^{(1)} 
                &=- \frac{1}{ \sigma_g} \mu \pdrv{(B_{tilt})_{g,i}^{(0)} }{x}+  B_{g,i}^{(1)}
                \\
                &=   B_{g,i}^{(1)} -
                \begin{cases} 
                  \frac{1}{ \sigma_g} \mu		\frac{B^{(0)}_{g,i  } - B^{(0)}_{g,i-1}}{ \Delta x} \left( x - x_i \right), & \text{if } \mu < 0, \\
                  \frac{1}{ \sigma_g} \mu		\frac{B^{(0)}_{g,i + 1} - B^{(0)}_{g,i}}{ \Delta x} \left( x - x_i \right), & \text{if } \mu > 0.
                \end{cases}
        \end{aligned}
    \end{equation*}
By integrating the $O(\varepsilon^2)$ equation \eqref{eq_asymptoticAnalysis3_O3A} over the cell $V_i$ and angle, then summing over all groups while applying \eqref{eq_asymptoticAnalysis3_O1B} and \eqref{eq_asymptoticAnalysis3_O3B}, we obtain
   \begin{equation}
        \frac{1}{c} \frac{\partial }{\partial t} \left( \sum_{g=1}^G 4 \pi B_{g,i}^{(0)} \right)+C_v \frac{\partial }{\partial t} T^{(0)} = -\frac{1}{\Delta x} \sum_{g=1}^G (F_{g,i+\frac{1}{2}}^{(1)} - F_{g,i-\frac{1}{2}}^{(1)}  ) , 
    \end{equation}
with the flux $F_{g,i+\frac{1}{2}}^{(1)}$ given by:
    \begin{equation*}
        \begin{aligned}
              F_{g,i+\frac{1}{2}}^{(1)} &= 2\pi \int_{-1}^{1} \mu I_{g,i+\frac{1}{2}}^{(1)} \dd{\mu}
              \\
              & =  2\pi \Big( \int_{-1}^{0} \frac{\mu^2}{ \sigma_{g,i+1}}  \pdrv{(B_{tilt})_{g,i+1}^{(0)}}{x} \dd{\mu}  +  \int_{0}^{1} \frac{\mu^2}{ \sigma_{g,i}}  \pdrv{(B_{tilt})_{g,i}^{(0)}  }{x} \dd{\mu}  \Big)
              \\
              & = - 2 \pi \Big( \int_{-1}^{0} \frac{ \mu^2}{\sigma_{g,i+1}}  \frac{B^{(0)}_{g,i + 1} - B^{(0)}_{g,i}}{\Delta x} \dd{\mu}  +  \int_{0}^{1} \frac{ \mu^2}{\sigma_{g,i }}  \frac{B^{(0)}_{g,i + 1} - B^{(0)}_{g,i}}{ \Delta x} \dd{\mu}  \Big)
              \\
              & = - \frac{4 \pi}{3 \sigma_{g,i+\frac{1}{2}}} \frac{B^{(0)}_{g,i + 1}- B^{(0)}_{g,i}}{\Delta x},
        \end{aligned}
    \end{equation*}
where $\sigma_{g,i+\frac{1}{2}}$ is the  harmonic average of  $\sigma_{g,i}$. Substituting the relation $\sum_{g=1}^G 4 \pi B_{g,i}^{(0)} = ac (T_{i}^{(0)} )^4$, we have 
    \begin{equation}
 	a \frac{\partial  }{\partial t} (T_{i}^{(0)} )^4 +C_v \frac{\partial }{\partial t} T_i^{(0)} = \frac{1}{\Delta x} \sum_{g=1}^G \left(\frac{4 \pi}{3 \sigma_{g,i+\frac{1}{2}}} \frac{B_{g,i + 1}^{(0)}- B_{g,i}^{(0)}}{\Delta x} - \frac{4 \pi}{3 \sigma_{g,i-\frac{1}{2}}} \frac{B_{g, i}^{(0)}- B_{g,i-1}^{(0)}}{\Delta x} \right).
 \end{equation}
    Therefore, a consistent discretization for \eqref{eq_def_diffLimit} is obtained.
\end{proof}

\section{Numerical Results}
\label{sec:results}
In this section, we present a series of numerical experiments to demonstrate the capability and effectiveness of the proposed method, termed Effective Monte Carlo (EMC). We compare EMC with the state-of-the-art Implicit Monte Carlo (IMC) method \cite{fleck1971implicit}. It is worth noting that while numerous acceleration techniques, variance reduction strategies, and code optimizations have been developed for IMC, our comparisons are limited to the version incorporating only \textit{continuous energy deposition} as described in \cite{fleck1971implicit}. For each numerical test, the IMC method employs the same time step, mesh size, and number of particles as the EMC method.

In the following examples, the units are defined as follows: length in centimeters (cm), time in nanoseconds (ns), temperature in kilo electron-volts (keV), and energy in $10^9$ Joules (GJ). Under these units, the speed of light $c$ is $29.98 \mathrm{~	cm} / \mathrm{ns}$ and the radiation constant $a$ is $0.01372 \, \mathrm{GJ} / \mathrm{cm}^3/ \mathrm{keV}^4$. All numerical tests use $2,000,000$ MC particles unless  otherwise specified. The reference CFL numerber is determined by 
\begin{equation*}
	 \text{CFL}=  \frac{c \Delta t }{\min{(\Delta x,\Delta y)}}.
\end{equation*}

\subsection{Infinite medium problem (Test of variance)} \label{Pb.1}
In this example, we compare the Figure of Merit (FOM) between the the EMC and IMC methods. The problem setup is similar to that described in \cite{mcclarren2009modified}. The FOM for Monte Carlo simulations is defined as
\begin{equation*}
    \text{FOM} = \frac{1}{\text{Var} \cdot t},
\end{equation*}
where $\text{Var}$ represents the variance of the estimate and $t$ denotes the CPU time. A higher FOM indicates better computational efficiency, corresponding to lower variance and reduced computation time.

We consider a steady-state, infinite medium problem with an initial equilibrium temperature of $T_{m, 0} = T_{r, 0} = 1.0 \, \text{keV}$. The system is modeled as a one-dimensional slab of thickness $1.0 \, \text{cm}$ with reflecting boundary conditions. The opacity is defined by
\begin{equation*}
	\sigma = \frac{300}{T^3}\,\text{cm}^{-1},
\end{equation*}
and the heat capacity is given by
\begin{equation*}
	C_v = 0.3 \, \text{GJ/keV/cm}^3.
\end{equation*}
The exact solution predicts that the medium remains at its initial temperature throughout the simulation. The computational domain is discretized into 50 uniform spatial cells, using a fixed time step of $\Delta t = 0.0025 \, \text{ns}$ ( $\text{CFL} \approx 15 $  ) .  The simulation is run until a final time of $t = 1.0 \,\text{ns}$.

The spatial variations of material and radiation temperatures at $t = 1.0\,\text{ns}$ are shown in \autoref{FOM_tempre} for both the EMC and IMC methods. It can be seen that the variations in the EMC method are more pronounced  than those in the IMC method. This is attributed to the absence of effective scattering in EMC. Nevertheless, the EMC method is significantly faster due to shorter particle lifetimes. To reach the simulation time of $t = 1.0\,\text{ns}$, the IMC method requires $669 \,\text{s}$, whereas the EMC method only takes $53 \,\text{s}$ as shown in \autoref{tabel_Pb1}. The FOMs for both methods are presented in \autoref{FOM}, showing that EMC achieves substantially higher FOMs for both material and radiation temperatures.

\begin{figure}[!htb] 
    \centering
    \includegraphics[width=0.5\textwidth]{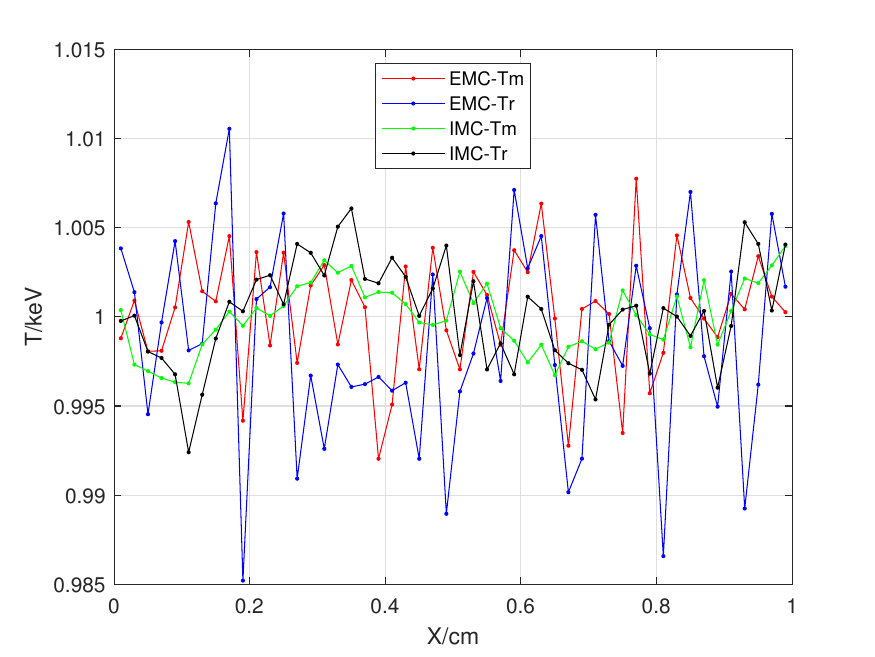}
    \caption{Comparison of the material and radiation temperatures using the EMC and IMC methods  at $t = 1.0\,\text{ns}$, with $\Delta t =0.0025\,\text{ns}$ ($\text{CFL} \approx 15 $) for the infinite medium problem.} 
    \label{FOM_tempre}
\end{figure}

\begin{figure}[!htb]
    \centering
    \includegraphics[width=0.5\textwidth]{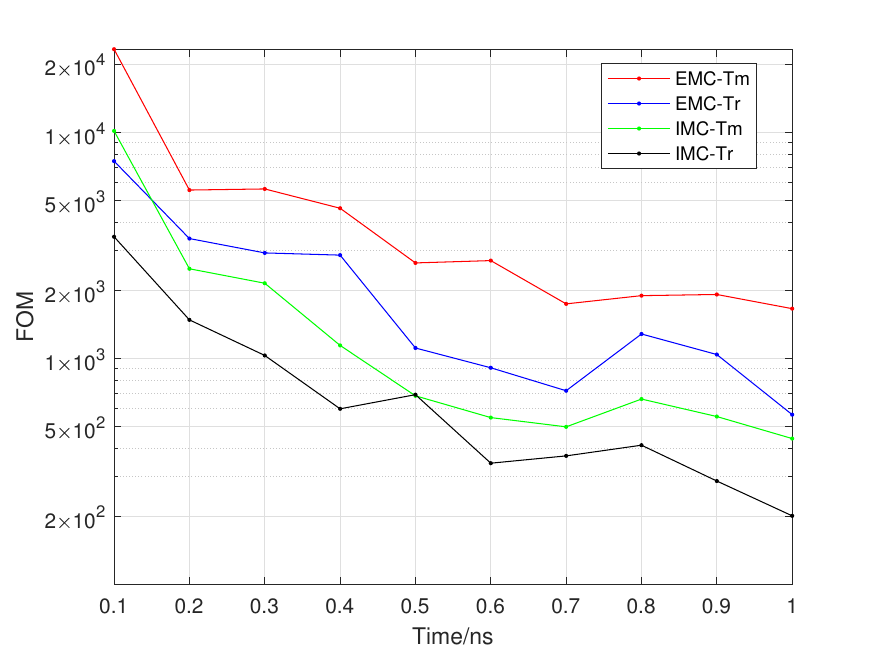}
    \caption{Figure of Merit for the material and radiation temperatures using the EMC and IMC methods.} 
    \label{FOM}
\end{figure}

\begin{table}[!htb]
	\centering
	\begin{tabular}{lrr}
		\toprule
		\textbf{Example \ref{Pb.1}} & \textbf{EMC (s)} & \textbf{IMC (s)} \\
		\midrule
		FOM & 54  & 606  \\
		\bottomrule
	\end{tabular}
	\caption{Comparison of CPU time using the EMC and IMC methods for Example \ref{Pb.1} (in seconds).}
	\label{tabel_Pb1}
\end{table}

\subsection{Marshark wave problems }\label{Pb.2}
For this example, we consider the frequency-dependent Marshak wave problems \cite{densmore2011asymptotic,steinberg2023frequency,sun2015asymptoticB,shi2023maximum}. These consist of several one-dimensional cases with varying optical depths: optically thin, optically thick, and a combination of both.

In all test problems, the initial temperature is in equilibrium, given by  $T_{r,0} = T_{m,0} = 10^{-3} \,$\text{keV}. The frequency-dependent opacity in each region is given by  
\begin{equation*}
    \sigma(x,\nu,T) = \frac{\sigma_0(x)}{(h\nu)^3\sqrt{kT}}\,\text{cm}^{-1},
\end{equation*}  
and the heat capacity is set to  
\begin{equation*}
    C_v = 0.1 \,\text{GJ/keV/cm}^3.
\end{equation*}  
To represent the frequency-dependent opacity, we employ 25 frequency groups spaced logarithmically between $10^{-3}\,\text{keV}$ and $100 \,\text{keV}$.

 At the left boundary, the incident intensity follows a Planckian distribution with a temperature of $1.0\,\text{keV}$, while a reflective boundary condition is applied at the right boundary. The simulation is run until a final time of $t = 1.0 \,\text{ns}$.

\subsubsection{Homogeneous problems}  
Two homogeneous test cases are considered in a computational domain of thickness $5.0\,\text{cm}$, with opacity values given by  
\begin{equation*}
    \sigma_0 = 10 \, \text{keV}^{7/2} / \text{cm}, \quad
    \sigma_0 = 1000 \, \text{keV}^{7/2} / \text{cm}.
\end{equation*}  
The spatial domain is discretized using a uniform mesh with a cell size of $\Delta x = 0.005$ cm.  The time step is set to $\Delta t = 0.0025 \,\text{ns}$ ( $\text{CFL} \approx 15 $ ). \autoref{Marshark_Ho_10} and \autoref{Marshark_Ho_1000} present the material and radiation temperatures computed using both the EMC and IMC methods. The results show good agreement between the two approaches. As shown in \autoref{Marshark_Ho_1000}, which corresponds to the optically thick case,  the temperature profiles are noisier. However, EMC demonstrates significantly better computational efficiency compared to IMC in the optically thick regime, as indicated in \autoref{tabel_Pb2}.

\begin{figure}[!htb]
    \centering
    \begin{subfigure}{0.49\textwidth}
        \centering
       \includegraphics[width=0.9\linewidth, height=6cm]{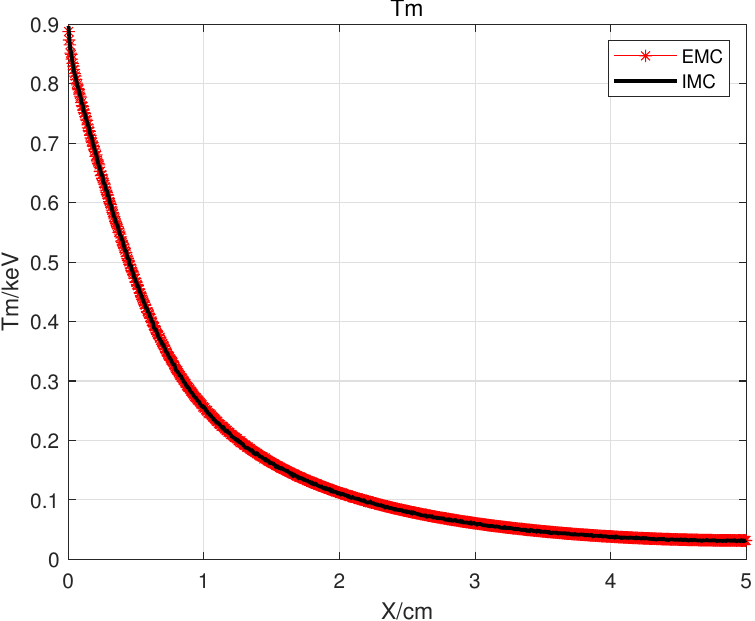}
        \subcaption{Material temperature.}
    \end{subfigure}
    \hfill
    \begin{subfigure}{0.49\textwidth}
        \centering
        \includegraphics[width=0.9\linewidth, height=6cm]{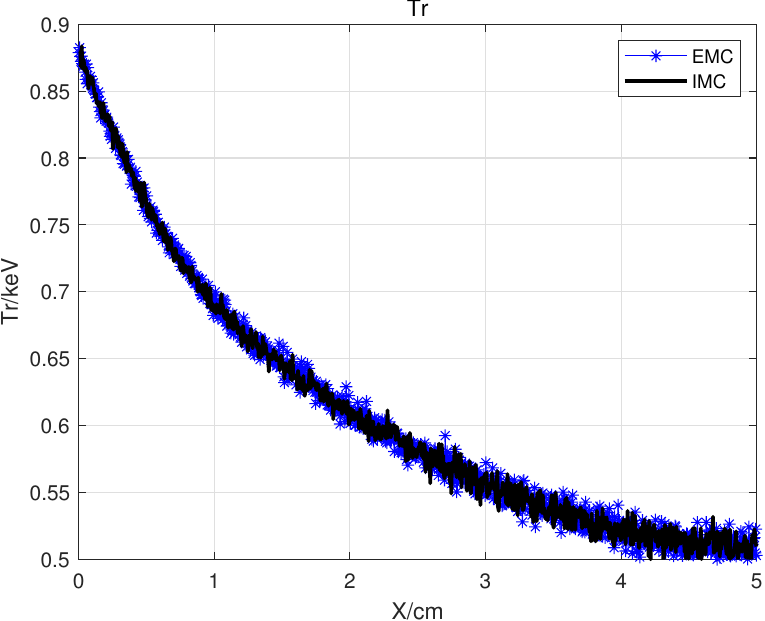}
        \subcaption{Radiation temperature.}
    \end{subfigure}
    \caption{Comparisons of the material and radiation temperatures using the EMC and IMC methods at $t = 1.0\,\text{ns}$, with $\Delta t =0.0025\,\text{ns}$ ($\text{CFL} \approx 15 $) for homogeneous  Marshark wave problem when $\sigma_0 = 10 \, \mathrm{keV}^{7 / 2} / \mathrm{cm}$.}
    \label{Marshark_Ho_10}
\end{figure}

\begin{figure}[!htb]
    \centering
    \begin{subfigure}{0.49\textwidth}
        \centering
        \includegraphics[width=0.9\linewidth, height=6cm]{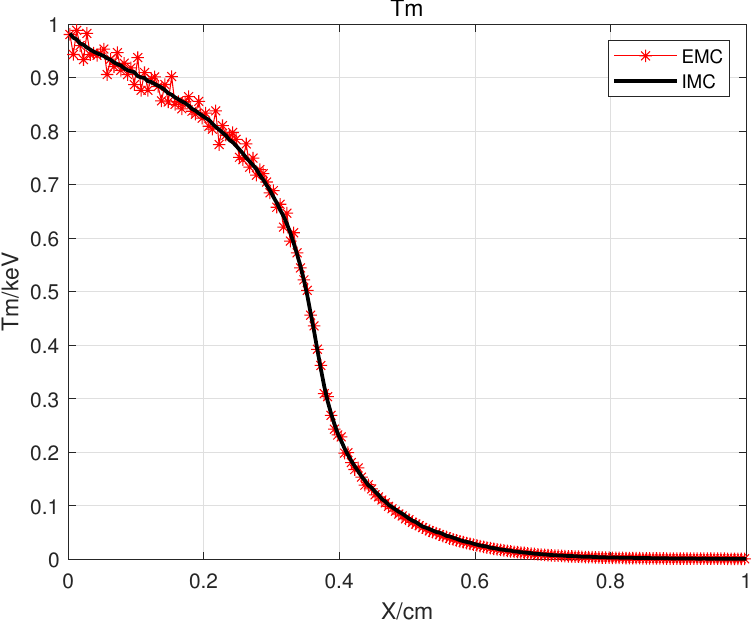}
        \subcaption{Material temperature.}
    \end{subfigure}
    \hfill
    \begin{subfigure}{0.49\textwidth}
        \centering
        \includegraphics[width=0.9\linewidth, height=6cm]{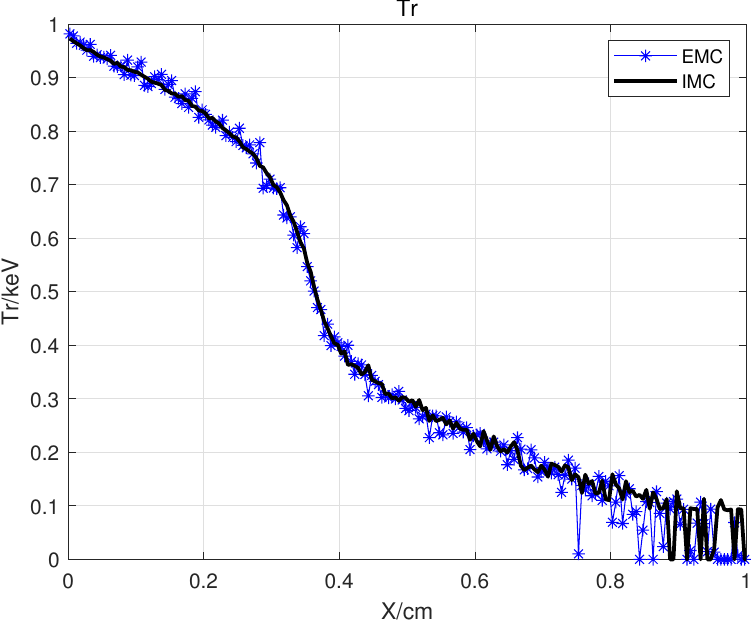}
        \subcaption{Radiation temperature.}
    \end{subfigure}
    \caption{Comparisons of the material and radiation temperatures using the EMC and IMC methods at $t = 1.0\,\text{ns}$, with $\Delta t =0.0025\,\text{ns}$ ($\text{CFL} \approx 15 $) for homogeneous  Marshark wave problem when $\sigma_0 = 1000 \, \mathrm{keV}^{7 / 2} / \mathrm{cm}$.}
    \label{Marshark_Ho_1000}
\end{figure}

\subsubsection{Heterogeneous problem A}  
The opacity profile in the computational domain is given by  
\begin{equation*}
    \sigma_0(x) = \begin{cases} 
        10 \,\text{keV}^{7/2} / \text{cm}, & 0 \,\text{cm} < x < 2 \,\text{cm}, \\
        1000 \,\text{keV}^{7/2} / \text{cm}, & 2 \,\text{cm} < x < 3 \,\text{cm}.
    \end{cases}
\end{equation*}  
The thickness of the computational domain is $3.0\,\text{cm}$, and the spatial mesh size is  
\begin{equation*}
    \Delta x = \begin{cases} 
        0.02 \,\text{cm}, & 0 \,\text{cm} < x < 2 \,\text{cm}, \\
        0.005 \,\text{cm}, & 2 \,\text{cm} < x < 3 \,\text{cm}.
    \end{cases}
\end{equation*}  
The time step is set to $\Delta t = 0.00125 \,\text{ns}$ ( $\text{CFL} \approx 8 $ ). This test problem evaluates the ability of our methods to handle a sharp transition from an optically thin to an optically thick regime. The simulation runs up to a final time of $1.0 \,\text{ns}$. \autoref{Marshark_He_A} presents the EMC and IMC results, which show good agreement. 
\begin{figure}[!htb]
    \centering
    \begin{subfigure}{0.49\textwidth}
        \centering
        \includegraphics[width=0.9\linewidth, height=6cm]{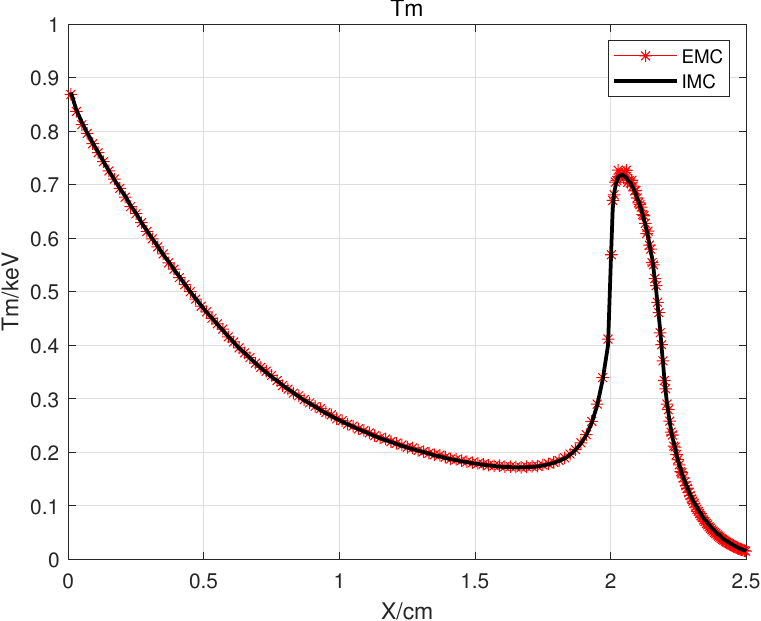}
        \subcaption{Material temperature.}
    \end{subfigure}
    \hfill
    \begin{subfigure}{0.49\textwidth}
        \centering
        \includegraphics[width=0.9\linewidth, height=6cm]{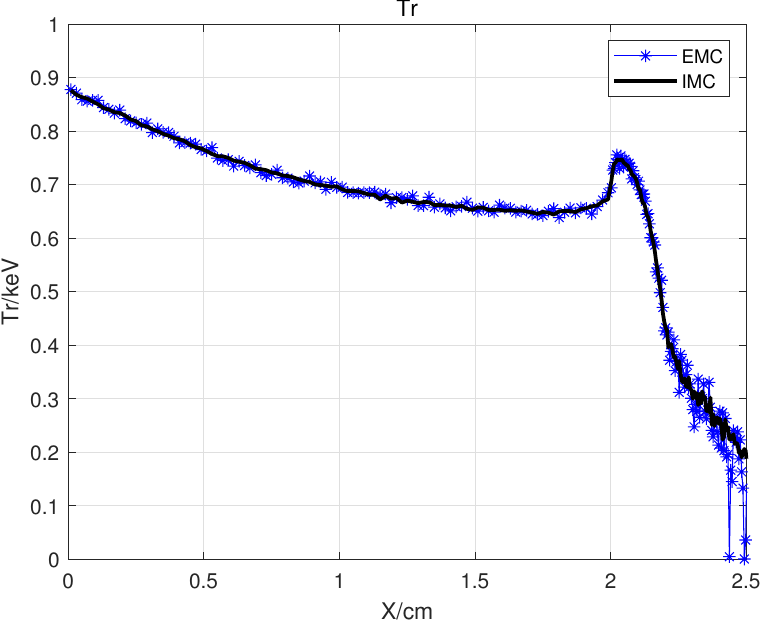}
        \subcaption{Radiation temperature.}
    \end{subfigure}
    \caption{Comparisons of the material and radiation temperatures using the EMC and IMC methods at $t = 1.0\,\text{ns}$, with $\Delta t =0.00125\,\text{ns}$ ($\text{CFL} \approx 8 $) for heterogeneous  Marshark wave problem A.}
    \label{Marshark_He_A}
\end{figure}

\subsubsection{Heterogeneous problem B}  
 The opacity profile in the computational domain is given by  
\begin{equation*}
    \sigma_0(x) = \begin{cases} 
        1000 \,\text{keV}^{7/2} / \text{cm}, & 0 \,\text{cm} < x < 0.5 \,\text{cm}, \\
        10 \,\text{keV}^{7/2} / \text{cm}, & 0.5 \,\text{cm} < x < 1.5 \,\text{cm}.
    \end{cases}
\end{equation*}  
The thickness of the computational domain is $1.5\,\text{cm}$, and the spatial mesh size is  
\begin{equation*}
    \Delta x = \begin{cases} 
        0.005 \,\text{cm}, & 0 \,\text{cm} < x < 0.5 \,\text{cm}, \\
        0.02 \,\text{cm}, & 0.5 \,\text{cm} < x < 1.5 \,\text{cm}.
    \end{cases}
\end{equation*}  
The time step is set to $\Delta t = 0.00125 \, \text{ns}$ (with $\text{CFL} \approx 8$). This problem assesses the capability of our methods to handle a sharp transition from an optically thick to an optically thin regime over a long simulation time.  The simulation runs up to a final time of $5.0\, \text{ns}$, making it particularly challenging due to the extended duration. Nevertheless, \autoref{Marshark_He_B} demonstrates strong agreement between the two approaches, especially near the thick-to-thin interface. Although the temperature exhibits more noise in the optically thick regime, \autoref{tabel_Pb2} demonstrates that EMC achieves excellent computational efficiency.

\begin{figure}[!htb]
	\centering
	\begin{subfigure}{0.49\textwidth}
		\centering
		\includegraphics[width=0.9\linewidth, height=6cm]{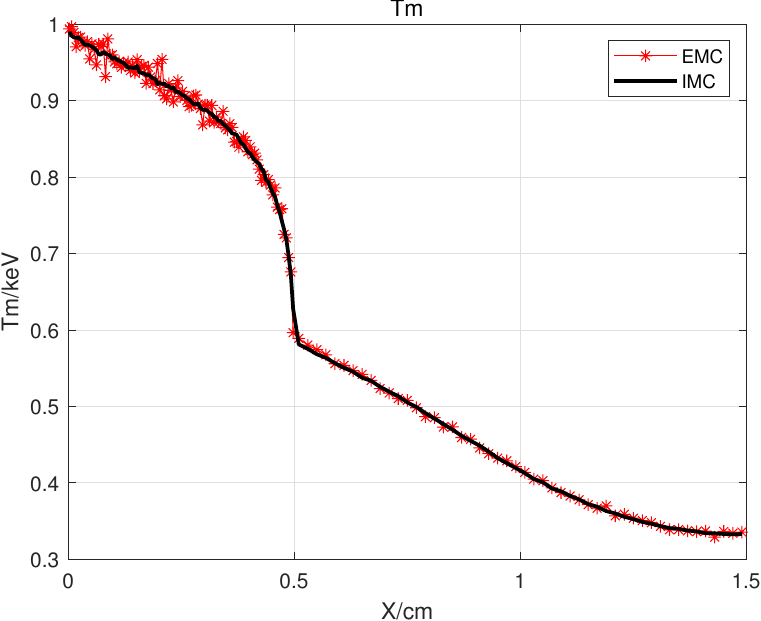}
		\subcaption{Material temperature.}
	\end{subfigure}
	\hfill
	\begin{subfigure}{0.49\textwidth}
		\centering
		\includegraphics[width=0.9\linewidth, height=6cm]{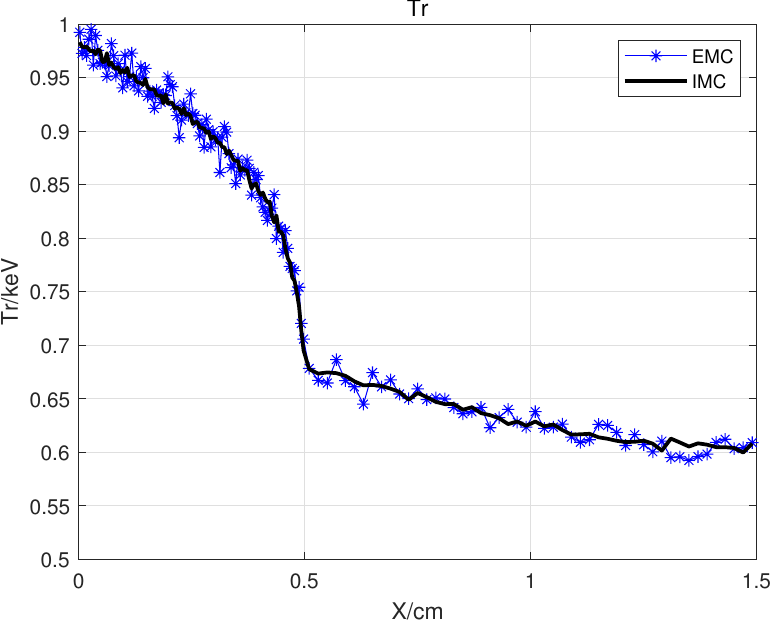}
		\subcaption{Radiation temperature.}
	\end{subfigure}
	\caption{Comparisons of the material and radiation temperatures using the EMC and IMC methods at $t = 1.0\,\text{ns}$, with $\Delta t =0.00125\,\text{ns}$ ($\text{CFL} \approx 8 $) for heterogeneous  Marshark wave problem B.}
	\label{Marshark_He_B}
\end{figure}

\begin{table}[!htb]
	\centering
	\begin{tabular}{lrr}
		\toprule
		\textbf{Example \ref{Pb.2}} & \textbf{EMC (s)} & \textbf{IMC (s)} \\
		\midrule
		$\sigma_0 = 10$        & 146   & 293    \\
		$\sigma_0 = 1000$      & 152   & 9125   \\
		Heterogeneous A        & 160   & 3008   \\
		Heterogeneous B        & 930   & 32762  \\
		\bottomrule
	\end{tabular}
	\caption{Comparison of CPU time using the EMC and IMC methods for Example \ref{Pb.2} (in seconds).}
	\label{tabel_Pb2}
\end{table}

    \subsection{Larsen’s problem }\label{Pb.3}
  Next, we consider the frequency-dependent, multi-material Larsen’s problem \cite{larsen1988grey, sun2015asymptoticA, shi2023maximum}. To model the frequency-dependent opacity, we employ 50 frequency groups spaced logarithmically between $10^{-5} \,\text{keV} $ and $10 \,\text{keV} $. The frequency-dependent opacity in each region is given by  
\begin{equation*}
    \sigma(x,\nu, T) = \sigma_0(x) \frac{1 - e^{-h \nu / k T}}{(h \nu)^3}\,\text{cm}^{-1},
\end{equation*}  
where   the spatially varying coefficient $\sigma_0(x)$ is given by
\begin{equation*}
    \sigma_0(x) = \begin{cases} 
        1 \,\text{keV}^{2/7} / \text{cm}, & 0 < x < 2 \,\text{cm}, \\  
        1000 \,\text{keV}^{2/7} / \text{cm}, & 2 < x < 3 \,\text{cm}, \\  
        1 \,\text{keV}^{2/7} / \text{cm}, & 3 < x < 4 \,\text{cm}.
    \end{cases}
\end{equation*}  
The heat capacity is specified as
\begin{equation*}
    C_v = 0.05109 \,\text{GJ/keV/cm}^3.
\end{equation*}  

The system is initialized in thermal equilibrium, with both radiation and material temperatures set to $T_{r,0} = T_{m,0} = 10^{-3} \,\text{keV} $. A Planckian surface source at $1.0 \,\text{keV} $  is applied at the left boundary to initiate the transient. The simulation runs until a final time of $0.9\,\text{ns}$. The spatial domain is divided into three regions with non-uniform mesh sizes:
\begin{equation*}
    \Delta x = \begin{cases} 
        0.2 \,\text{cm}, & 0 < x < 2 \,\text{cm}, \\  
        0.02 \,\text{cm}, & 2 < x < 3 \,\text{cm}, \\  
        0.1 \,\text{cm}, & 3 < x < 4 \,\text{cm}.
    \end{cases}
\end{equation*}  
The fixed time step $\Delta t = 0.005\,\text{ns}$ ($  \text{CFL} \approx 8 $) is used throughout the simulation. 

This problem evaluates the ability of our methods to accurately capture the sharp transition from an optically thick to an optically thin regime. The simulation runs up to a final time of $0.9\, \text{ns}$. Nevertheless, \autoref{Larsen} demonstrates strong agreement between the two approaches, especially near the thick-to-thin interface. \autoref{tabel_Pb3} demonstrates that EMC achieves excellent computational efficiency.

\begin{figure}[!htb]
    \centering
    \begin{subfigure}{0.49\textwidth}
        \centering
        \includegraphics[width=0.9\linewidth, height=6cm]{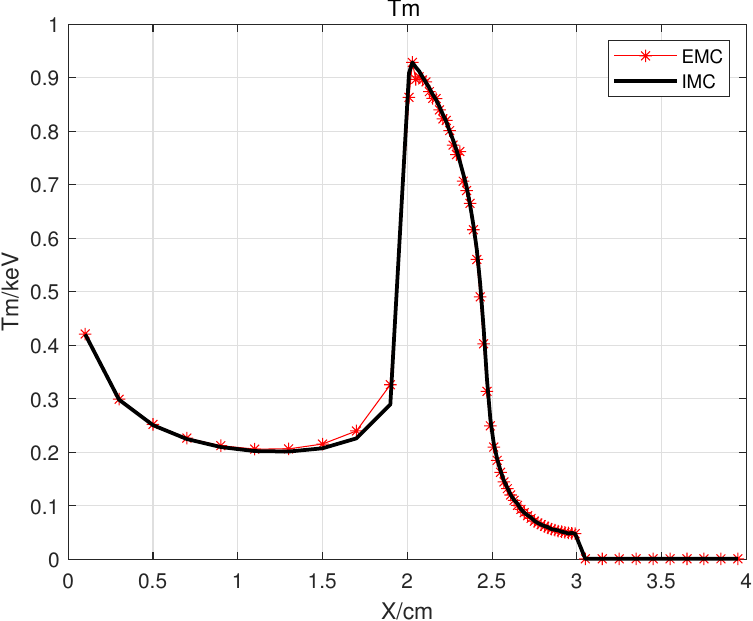}
        \subcaption{Material temperature.}
    \end{subfigure}
    \hfill
    \begin{subfigure}{0.49\textwidth}
        \centering
        \includegraphics[width=0.9\linewidth, height=6cm]{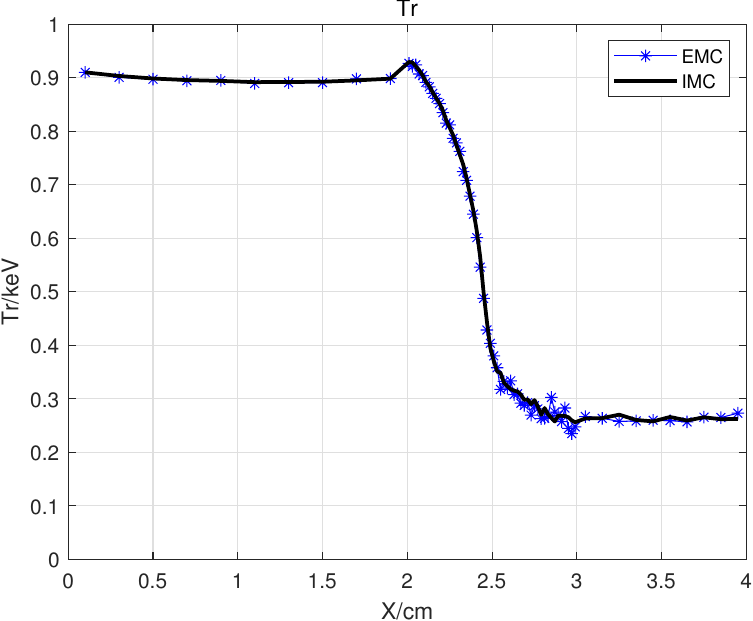}
        \subcaption{Radiation temperature.}
    \end{subfigure}
    \caption{Comparisons of the material and radiation temperatures using the EMC and IMC methods at $t = 0.9\,\text{ns} $, with $\Delta t =0.005\,\text{ns} $ ($  \text{CFL} \approx 8 $)  for Larsen’s problem.}
    \label{Larsen}
\end{figure}

\begin{table}[!htb]
	\centering
	\begin{tabular}{lrr}
		\toprule
		\textbf{Example \ref{Pb.3}} & \textbf{EMC (s)} & \textbf{IMC (s)} \\
		\midrule
		Larsen's Pb. & 24  & 404  \\
		\bottomrule
	\end{tabular}
		\caption{Comparison of CPU time using the EMC and IMC methods for Example \ref{Pb.3} (in seconds).}
	\label{tabel_Pb3}
\end{table}

\subsubsection{Test of different weight functions} \label{weightFunction_compared}
We use Larsen's problem to to demonstrate the necessity of modifying the original model using \eqref{eq_formal_solution2} and the fact that the proposed scheme is generally insensitive to the specific choice of the weight function $\theta_g$. The candidate weight functions considered are $\theta_g = e^{-c\sigma_g (t - t^n)}$ and $\theta_g = 1 - e^{-1 / \left(c\sigma_g (t - t^n) \right)}$, referred to as EMC1 and EMC2, respectively.  When the time step is  $\Delta t = 0.005\,\text{ns}$ ($  \text{CFL} \approx 8 $), the results shown in \autoref{Larsen_compared} indicate that the solutions obtained with EMC1 and EMC2 are in close agreement. In contrast, using $\theta_g = 0$, which  corresponds to the unmodified model given in \eqref{eq_approximaed_nonModifed}, leads to failure in the convergence of the nonlinear iteration for the macroscopic system.

\begin{figure}[!htb]
    \centering
    \begin{subfigure}{0.49\textwidth}
        \centering
        \includegraphics[width=0.9\linewidth, height=6cm]{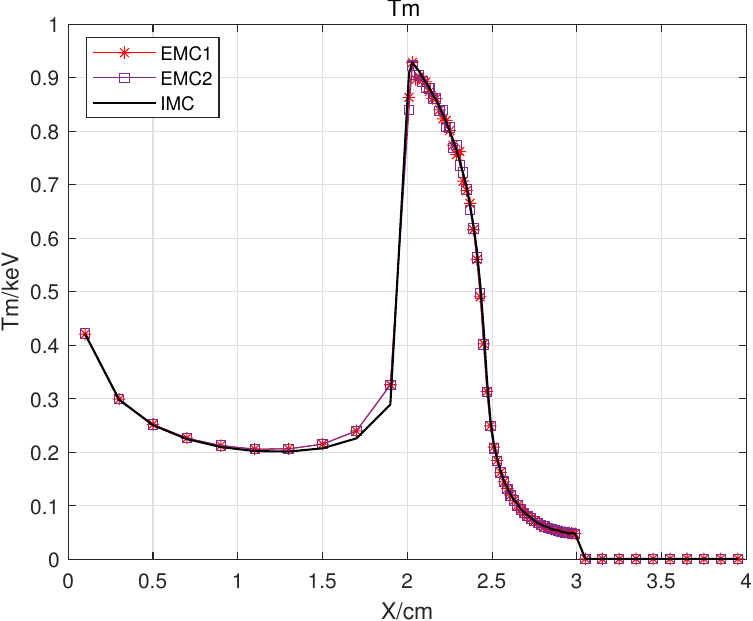}
        \subcaption{Material temperature.}
    \end{subfigure}
    \hfill
    \begin{subfigure}{0.49\textwidth}
        \centering
        \includegraphics[width=0.9\linewidth, height=6cm]{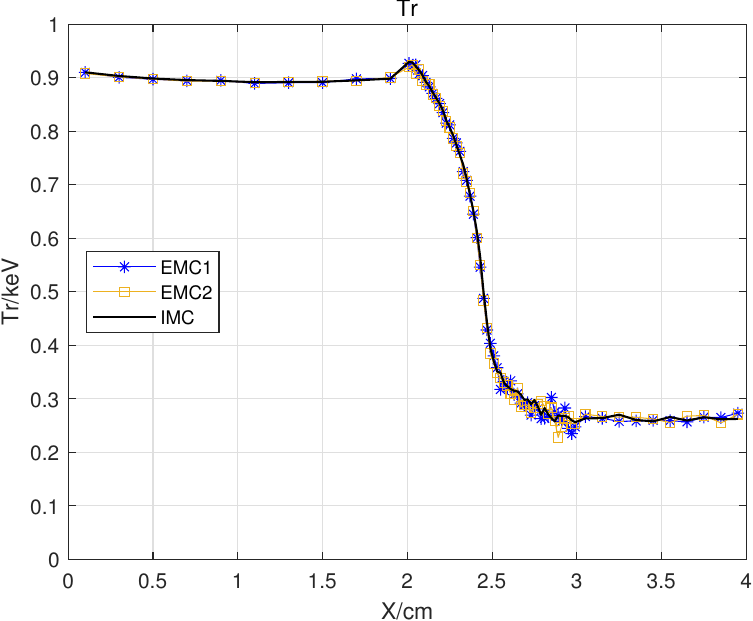}
        \subcaption{Radiation temperature.}
    \end{subfigure}
    \caption{Comparisons of the material and radiation temperatures using the EMC method with different weight functions and the IMC method at $t = 0.9\,\text{ns} $, with $\Delta t =0.005\,\text{ns} $ ($  \text{CFL} \approx 8 $)  for Larsen’s problem.}
    \label{Larsen_compared}
\end{figure}

 \subsection{Frequency-dependent hohlraum problem}\label{Pb.4}
 \begin{figure}[!htb]
 	\centering
 	\includegraphics[width=0.5\textwidth]{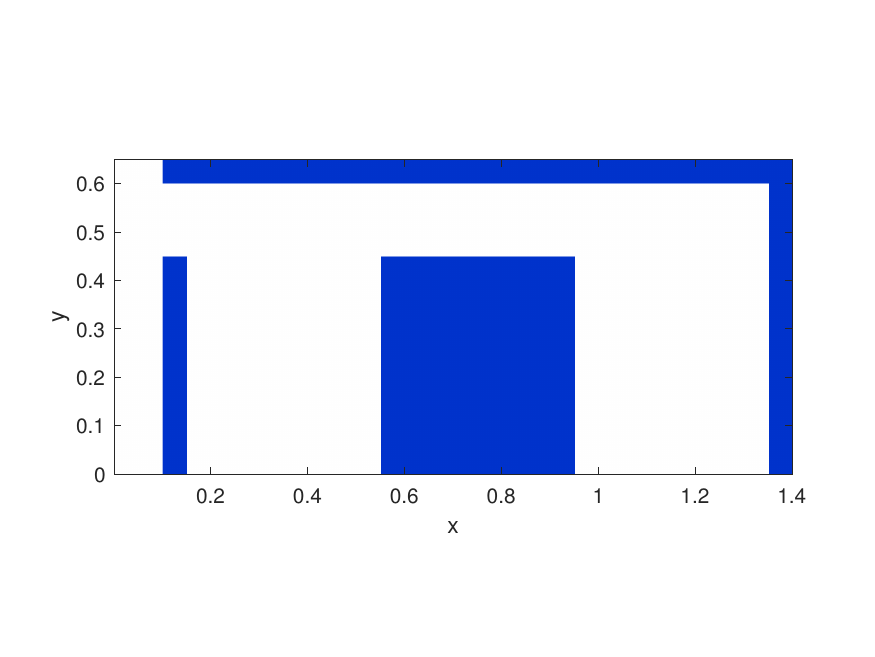}
 	\caption{The hohlraum problem. The blue regions are where $(x, y) \in[0.1,0.15] \times[0,0.45]$, and $(x, y) \in[0.55,0.95] \times[0,0.45],(x, y) \in[0.1,1.4] \times[0.6,0.65]$, and $(x, y) \in[1.35,1.4] \times[0,0.65]$.}
 \label{FRTE_hohlraum_layout}
 \end{figure}
 
For the final example, we study the hohlraum problem for the frequency-dependent radiative transfer equations. The setup of this problem is similar as that studied in \cite{li2024unified,Hammer2019}. To represent the frequency-dependent opacity, our method employs 50 frequency groups spaced logarithmically between $10^{-5} \,\text{keV} $ and $10 \,\text{keV} $.

The layout of the problem is illustrated in \autoref{FRTE_hohlraum_layout}. The computational domain spans  
$[0\,\text{cm}, 1.4\,\text{cm}] \times [0\,\text{cm}, 0.65\,\text{cm}]$,
where the white regions represent near-vacuum. We assume an absorption coefficient of  
\begin{equation*}
    \sigma = 10^{-8} \,\text{cm}^{-1},
\end{equation*}  
and a specific heat capacity of  
\begin{equation*}
    C_v = 10^{-4} \,\text{GJ/keV/cm}^3,
\end{equation*}  
for the white regions. The blue regions are filled with material that follows the frequency-dependent opacity relation  
\begin{equation*}
    \sigma(x,\nu, T) = 1000 \frac{1 - e^{-h \nu / k T}}{(h \nu)^3}\,\text{cm}^{-1},
\end{equation*}  
with a specific heat capacity of    
\begin{equation*}
    C_v = 0.3 \,\text{GJ/keV/cm}^3.
\end{equation*}  
The initial temperature is in equilibrium, given by $T_{r,0} = T_{m,0} = 10^{-3} \,\text{keV} $. A reflective boundary condition is imposed on the lower boundary. The left boundary is maintained with an angularly isotropic specific intensity corresponding to a $ 0.3 \,\text{keV} $ black body source. The upper and right boundaries are fixed at a specific intensity described by a Planckian distribution with temperature $ 10^{-3} \,\text{keV} $.  The time step is set to $\Delta t = 0.0025 \, \text{ns}$ ($  \text{CFL} \approx 12 $). A total of 6, 000, 000 particles is employed per time step. In particular, we use the arithmetic average to evaluate $\sigma_{ij}$ at the material interface, instead of using the harmonic average given in \eqref{eq_eva_sigma}. 

In \autoref{eq_holo}, we present the radiation and material temperatures at time $t = 10.0  \, \text{ns}$  obtained from the IMC and EMC solutions. We observe that the central block is heated non-uniformly, and the EMC solution is free from ray effects. The comparisons of material temperature along the diagnostic lines $y = 0.45 \, \text{cm}  $ and $y = 0.65 \, \text{cm} $ are shown in \autoref{eq_holo_cut}, where the IMC and EMC results are generally consistent with each other. As reported in \autoref{tabel_Pb4}, although the EMC solution is noisier, it requires only 17,687 seconds of CPU time to reach $t = 10.0 , \text{ns}$, whereas IMC takes 96,332 seconds, making EMC about five times faster in this case.

\begin{figure}[!htb]
   \centering
   \begin{subfigure}{0.43\textwidth}
       \centering
       \includegraphics[width=\textwidth]{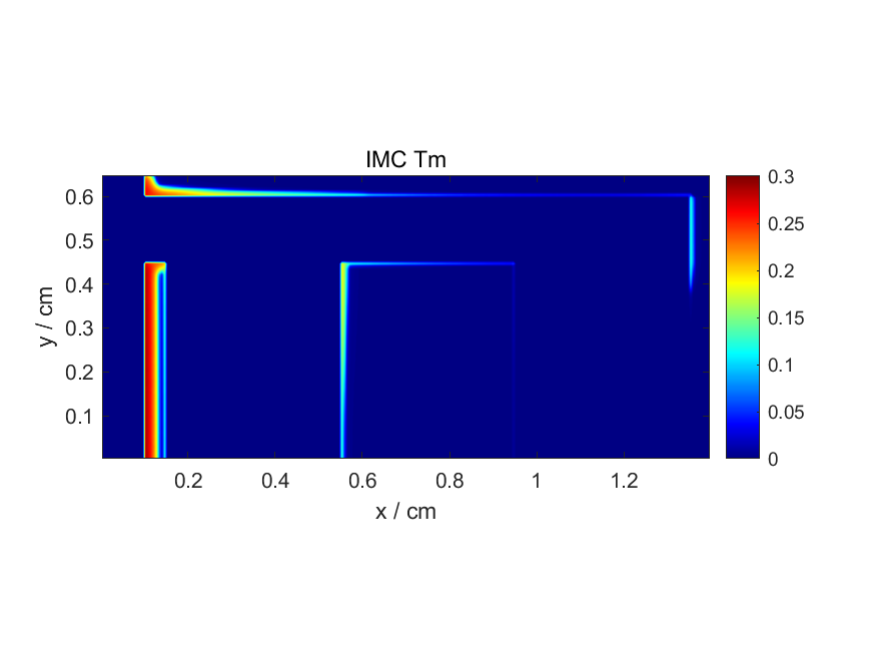}
       \subcaption{Material temperature  using  the IMC method.}
   \end{subfigure}
   \hfill
   \begin{subfigure}{0.43\textwidth}
       \centering
       \includegraphics[width=\textwidth]{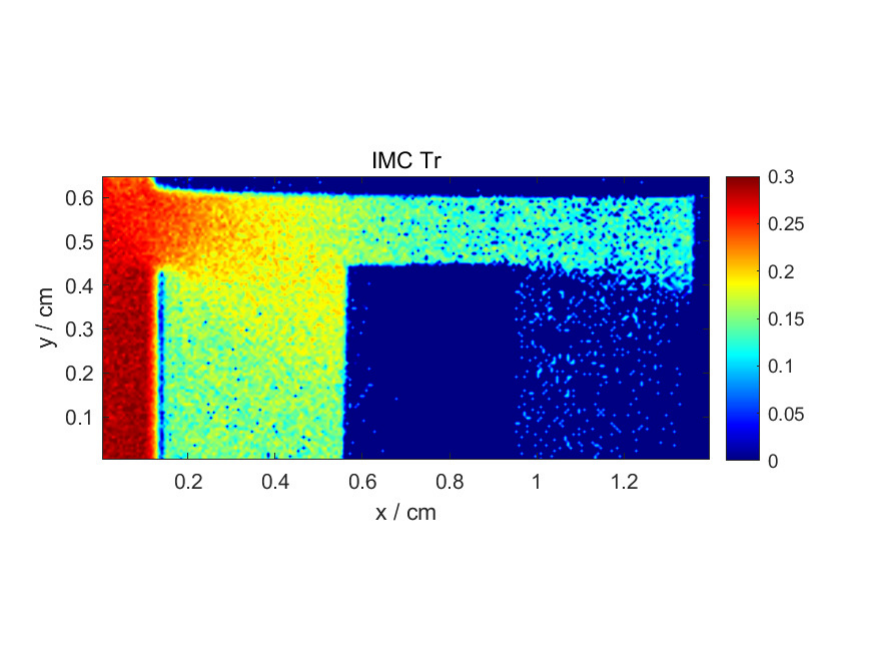}
       \subcaption{Radiation temperature using  the IMC method.}
   \end{subfigure}

   \vspace{-0.1cm} 

   \begin{subfigure}{0.43\textwidth}
       \centering
       \includegraphics[width=\textwidth]{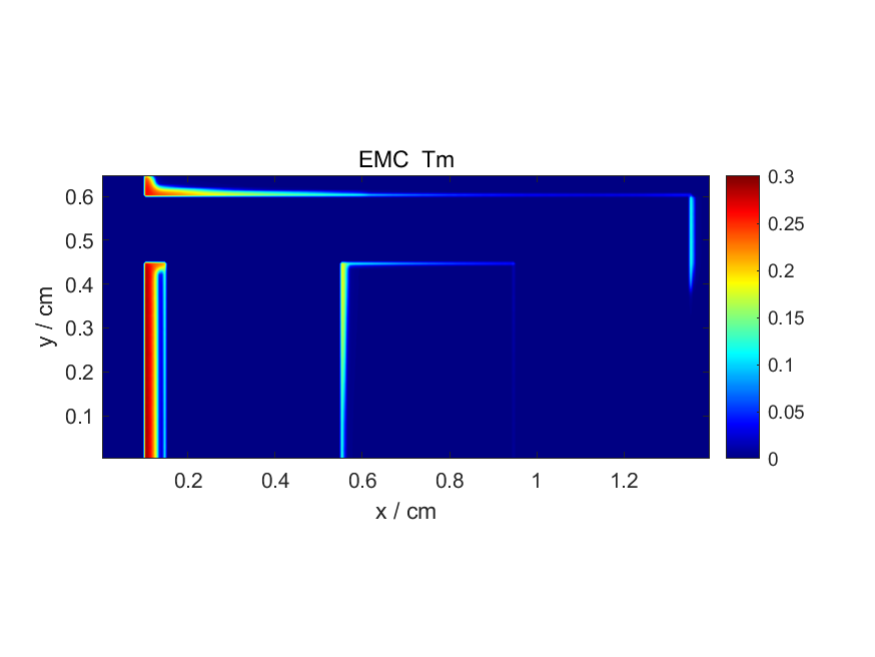}
       \subcaption{Material temperature using the EMC method.}
   \end{subfigure}
   \hfill
   \begin{subfigure}{0.43\textwidth}
       \centering
       \includegraphics[width=\textwidth]{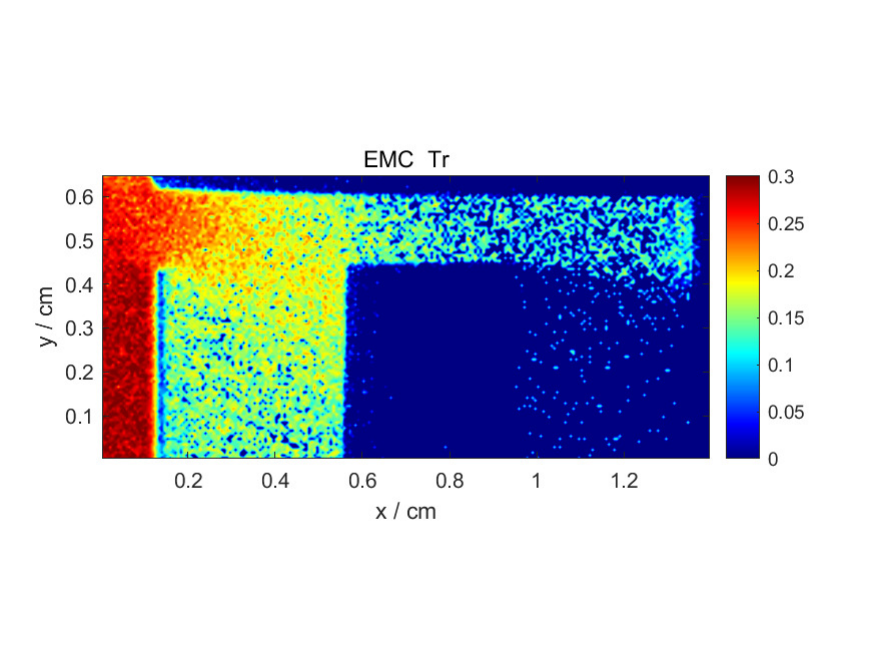}
       \subcaption{Radiation temperature using the EMC method.}
   \end{subfigure}

   \caption{Comparisons of the material and radiation temperatures using the EMC and IMC methods at $t = 10.0\,\text{ns} $, with  $\Delta t = 0.0025\,\text{ns}$ ($  \text{CFL} \approx 12 $) for the frequency-dependent hohlraum problem.}
   \label{eq_holo}
\end{figure}

\begin{figure}[!htb]
   \centering
   \begin{subfigure}{0.45\textwidth}
       \centering
       \includegraphics[width=\textwidth]{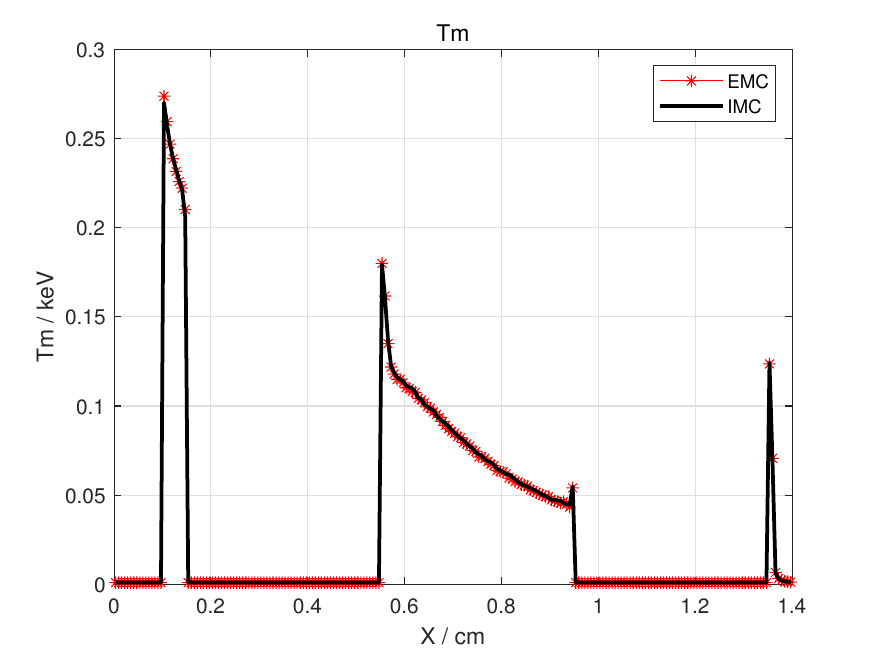}
       \subcaption{Material temperature along the line $y = 0.45\,\text{cm}$.}
   \end{subfigure}
   \hfill
   \begin{subfigure}{0.45\textwidth}
       \centering
       \includegraphics[width=\textwidth]{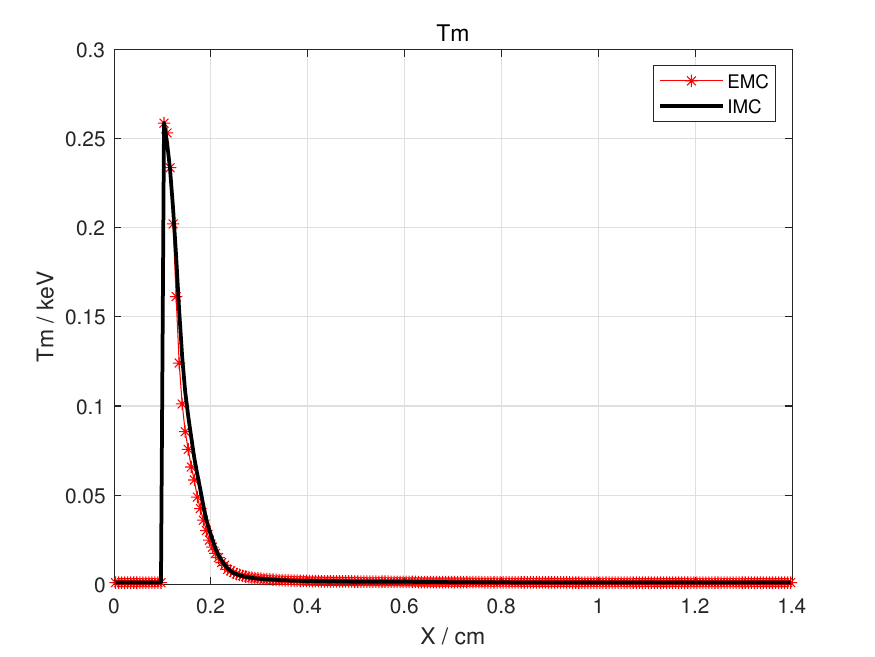}
       \subcaption{Material temperature along the line $y = 0.65\,\text{cm}$.}
   \end{subfigure}
   \caption{Comparisons of the material temperature using the EMC and IMC methods at $t = 10.0\,\text{ns} $ for the frequency-dependent hohlraum problem.}
      \label{eq_holo_cut}
\end{figure}

\begin{table}[!htb]
	\centering
	\begin{tabular}{lrr}
		\toprule
		\textbf{Example \ref{Pb.4}} & \textbf{EMC (s)} & \textbf{IMC (s)} \\
		\midrule
		Hohlraum  Pb. &  17687 & 96332 \\
		\bottomrule
	\end{tabular}
		\caption{Comparison of CPU time using the EMC and IMC methods for Example \ref{Pb.4} (in seconds).}
	\label{tabel_Pb4}
\end{table}

\section{Conclusions and outlook}
\label{sec:conclusion}

In this work, we develop an efficient AP MC method for frequency-dependent radiative transfer equations. By combining a multi-group frequency discretization with characteristic-based flux construction, we derive a micro-macro system which couples a low dimension convection-diffusion-type equation for macroscopic quantities with a high dimension transport equation for radiative intensity. This formulation enables the use of large time steps independent of the speed of light. A hybrid finite volume scheme is employed to efficiently solve the macroscopic system, while a Picard iteration with a predictor-corrector strategy effectively manages the high-dimensional nonlinear coupling across both spatial and frequency dimensions. The resulting transport problem reduces to a tractable absorption-only system, which is solved using a particle-based MC method. The scheme has been formally proved to be AP. Numerical results confirm substantial efficiency gains over the IMC method, especially in optically thick regime. 

While our new approach requires much less CPU time, the absence of effective scattering leads to considerably higher noise compared to IMC under the same settings. Variance reduction techniques for the proposed method will be explored in future work.


\section*{Acknowledgements}
We would like to thank Weiming Li and Chang Liu (Institute of Applied Physics and Computational Mathematics) and Xiaojiang Zhang  (Central South University) for helpful discussions.


\appendix
\section{Newton iteration for the correction step}  \label{appendix_A}
	In the correction step \eqref{eq_corretion}, we are required to solve the nonlinear equation $\mathcal{F}(T^{n+1,k+1}_i) = 0$ for each cell, where the function $\mathcal{F}(T^{n+1,k+1}_i)$ is defined as
\begin{equation}
	\begin{aligned}
		\mathcal{F}(T^{n+1,k+1}_i) := T_i^{n+1,k+1} &+ \frac{ a}{C_{v,i}} \sum_{g=1}^G \chi_{g,i}^{n+1,k+\half} b_{g,i}^{n+1,k+1}  (T_i^{n+1,k+1})^4 -A_i,
	\end{aligned}
\end{equation}
	with
\begin{equation}
	\begin{aligned}
		A_i  =&   T_i^n + \frac{\Delta t}{C_{v,i}} \sum_{g=1}^G \chi_{g,i}^{n+1, k+ \half}\left( \frac{1}{c \Delta t} \rho_{g,i}^n- \frac{1}{\Delta V_i}  \sum_{j \in \mathcal{N}_i} 
		F_{g,ij}^{n+1,k+\half*} \right).
	\end{aligned}
\end{equation}
The derivative of $\mathcal{F}(T^{n+1,k+1}_i)$ is given by 
\begin{equation}  
            \mathcal{F}^{'}(T^{n+1,k+1}_i) = 1 + \frac{ 4a}{C_{v,i}} \sum_{g=1}^G \chi_{g,i}^{n+1,k+\half} \Big( b_g+\frac{T}{4}\frac{\partial b_g}{\partial T} \Big)_{i}^{n+1,k+1}   (T_{i}^{n+1,k+1})^3.
\end{equation}
To solve for $T^{n+1,k+1}_i$, we employ the Newton iteration method:
\begin{equation}
    T^{n+1,k+1,s+1}_i =  T^{n+1,k+1,s}_i - \frac{\mathcal{F}(T^{n+1,k+1,s}_i)}{\mathcal{F}^{'}(T^{n+1,k+1,s}_i)},
\end{equation}
where the initial guess $T_i^{n+1,k+1,0}$ is taken from $T_i^{n+1,k+\frac{1}{2}}$ obtained in the prediction step.

\section{Proof of Proposition \ref{eq_prpo1} }

\begin{proof}
    Consider the systems
           \begin{subequations}
		\begin{align}
        \label{eq_asymptoticAnalysis1A}
            &\begin{aligned}
                \frac{\varepsilon^2}{c} \frac{\partial I_g}{\partial t}+\varepsilon \vOmega \cdot \nabla I_g = \sigma_{g}B_g-\sigma_{g}I_g,  \quad g=1,\ldots,G, 
            \end{aligned}
            \\
            \label{eq_asymptoticAnalysis1B}
		  &\begin{aligned}
                \varepsilon^2 C_v \frac{\partial T}{\partial t}=\sum_{g=1}^{G} \sigma_{g} \left( \rho_g- 4 \pi B_g  \right). 
            \end{aligned}
		\end{align}
   \end{subequations}
   We now perform a Chapman-Enskog expansion and compare terms that are the same order in $\varepsilon$.
   
    The $O(1)$ equation for \eqref{eq_asymptoticAnalysis1A} is
    \begin{equation}\label{eq_asymptoticAnalysis1_O1A}
        I_g^{(0)} = B_g^{(0)},  \quad g=1,\ldots,G, 
    \end{equation}
    integrating \eqref{eq_asymptoticAnalysis1_O1A} over the angular variables, we have 
    \begin{equation}\label{eq_asymptoticAnalysis1_O1B}
        \rho_g^{(0)} = 4 \pi B_g^{(0)},  \quad g=1,\ldots,G. 
    \end{equation}
    The $O(\varepsilon)$ equation for \eqref{eq_asymptoticAnalysis1A} is
    \begin{equation}\label{eq_asymptoticAnalysis1_O2A}
        \vOmega \cdot \nabla I_g^{(0)}  + \sigma_{g}I_g^{(1)} = \sigma_{g} B_g^{(1)} ,  \quad g=1,\ldots,G, 
    \end{equation}
    substituting \eqref{eq_asymptoticAnalysis1_O1A} into \eqref{eq_asymptoticAnalysis1_O2A}, we can get
    \begin{equation}\label{eq_asymptoticAnalysis1_O2B}
        \begin{aligned}
                I_g^{(1)} = - \frac{1}{ \sigma_g} \vOmega \cdot \nabla B_g^{(0)}+  B_g^{(1)},  \quad g=1,\ldots,G. 
        \end{aligned}
    \end{equation}
    The $O(\varepsilon^2)$ equation for \eqref{eq_asymptoticAnalysis1A} is
    \begin{equation}\label{eq_asymptoticAnalysis1_O3A}
        \frac{1}{c} \frac{\partial I_g^{(0)}}{\partial t}+\vOmega \cdot \nabla I_g^{(1)}  = \sigma_{g} B_g^{(2)} -\sigma_{g}I_g^{(2)},  \quad g=1,\ldots,G, 
    \end{equation}
        The $O(\varepsilon^2)$ equation for \eqref{eq_asymptoticAnalysis1B} is
    \begin{equation}\label{eq_asymptoticAnalysis1_O3B}
        C_v \frac{\partial }{\partial t}T^{(0)}=\sum_{g=1}^{G} \sigma_{g} \left( \rho_g^{(2)}- 4 \pi B_g^{(2)}  \right). 
    \end{equation}
    Integrating \eqref{eq_asymptoticAnalysis1_O3A} over the angular variables, adding up all groups and using \eqref{eq_asymptoticAnalysis1_O3B}, we have
    \begin{equation}\label{eq_asymptoticAnalysis1_O4}
        \frac{1}{c} \frac{\partial }{\partial t} \left( \sum_{g=1}^{G} \rho_g^{(0)} \right) + C_v \frac{\partial }{\partial t}T^{(0)}  
        = - \sum_{g=1}^{G} \nabla \cdot \int_{4 \pi} \vOmega I_g^{(1)} \dd{\vOmega}.
    \end{equation}
    Plugging \eqref{eq_asymptoticAnalysis1_O2B} into the equation \eqref{eq_asymptoticAnalysis1_O4}, using the condition \eqref{eq_asymptoticAnalysis1_O1B}, then \eqref{eq_asymptoticAnalysis1_O4} reduces to
        \begin{equation}
        \frac{1}{c} \frac{\partial }{\partial t} \left( 4 \pi \sum_{g=1}^{G} B_g^{(0)} \right) + C_v \frac{\partial }{\partial t}T^{(0)}  
        = \sum_{g=1}^{G} \nabla \cdot \left( \frac{4 \pi}{3 \sigma_g} \nabla B_g^{(0)} \right),
    \end{equation}
    which implies that (by chain rules)
    \begin{equation}\label{eq_asymptoticAnalysis1_O5}
        \frac{1}{c} \frac{\partial }{\partial t} \left( 4 \pi \sum_{g=1}^{G} B_g^{(0)} \right) + C_v \frac{\partial }{\partial t}T^{(0)}  
        = \sum_{g=1}^{G} \nabla \cdot \left( \frac{ac}{3 \sigma_g} \pdrv{B_g^{(0)}}{T} \frac{4 \pi}{4ac (T^{(0)})^3} \nabla  (T^{(0)})^4 \right).
    \end{equation}
   Using the following relations
    \begin{equation*}
        4 \pi \sum_{g=1}^{G} B_g  = ac T^4, \quad 4 \pi \sum_{g=1}^{G} \pdrv{B_g}{T}   = 4 ac T^3,
    \end{equation*}
    \eqref{eq_asymptoticAnalysis1_O5} reduces to
    \begin{equation}
        a \frac{\partial }{\partial t} (T^{(0)})^4 + C_v \frac{\partial }{\partial t}T^{(0)}  
        =  \nabla \cdot \left( \frac{ac}{3 \hat{\sigma}_R}  \nabla  (T^{(0)})^4 \right),
    \end{equation}
    where the mean opacity $\hat{\sigma}_R$ is defined by 
        \begin{equation*}
        \frac{1}{\hat{\sigma}_R} 
    = \frac{1}{ \sum_{g=1}^{G} \pdrv{B_g^{(0)}}{T} }  \left( \sum_{g=1}^{G} \frac{1}{  \sigma_g }  \pdrv{B_g^{(0)}}{T}   \right) .
    \end{equation*}
    When one approximates $\sigma_g$ as in \eqref{eq_Approximation_sigma}, i.e.,
    \begin{equation*}
        \sigma_g = \frac{1}{\nu_{g+\half}-\nu_{g-\half}} \int_{\nu_{g-\half}}^{\nu_{g+\half}} \sigma  \dd{\nu},
    \end{equation*}
    the approximated mean opacity $\hat{\sigma}_R $ of the above equation is determined by
    \begin{equation}
        \frac{1}{\hat{\sigma}_R} 
      = \frac{1}{ \sum_{g=1}^{G} \int_{\nu_{g-\half}}^{\nu_{g+\half}} \pdrv{B(\nu,T^{(0)})}{T}  \dd{\nu} }  \left( \sum_{g=1}^{G} \frac{1}{\frac{1}{\nu_{g+\half}-\nu_{g-\half}} \int_{\nu_{g-\half}}^{\nu_{g+\half}} \sigma  \dd{\nu} } \int_{\nu_{g-\half}}^{\nu_{g+\half}} \pdrv{B(\nu,T^{(0)})}{T}  \dd{\nu} \right),
    \end{equation}
    which is indeed a reasonable approximation for \eqref{eq_def_Rosseland}.
\end{proof}

\bibliographystyle{elsarticle-num}
\bibliography{src/references}

\end{document}